\documentclass[12pt]{article}

\usepackage[utf8]{inputenc}

\usepackage{amsthm}
\usepackage{tensor}
\usepackage{mathtools}
\usepackage{amsfonts}
\usepackage{amsmath}
\usepackage{amssymb}
\usepackage{upgreek}
\usepackage{multirow}
\usepackage{mathtools}
\usepackage{enumerate}
\usepackage{listings}
\usepackage{authblk}

\usepackage{tikz}
\usepackage{tikz-cd}
\usetikzlibrary{matrix,arrows}
\newtheorem{theorem}{Theorem}

\newtheorem{observation}{Observation}
\newtheorem{definition}{Definition}
\newtheorem{proposition}{Proposition}
\newtheorem{characterization}{Characterization}
\newtheorem{remark}{Remark}
\newtheorem{lemma}{Lemma}
\newtheorem{corollary}{Corollary}

\begin{document}

\title{Nil-Killing vector fields and type III deformations}
\author{Matthew Terje Aadne}
\affil{Faculty of Science and Technology,\\
University of Stavanger,\\
4036 Stavanger, Norway}
\affil{matthew.t.aadne@uis.no}
\maketitle
\begin{abstract}
    This paper is concerned with deformations of Kundt metrics in the direction of type $III$ tensors and nil-Killing vector fields whose flows give rise to such deformations. We find various characterizations within the Kundt class in terms of nil-Killing vector fields and obtain a theorem classifying algebraic stability of tensors, which has an application in finding sufficient criteria for a type $III$ deformation of the metric to preserve spi's. This is used in order to specify  Lie algebras of nil-Killing vector fields that preserve the spi's, for degenerate Kundt metrics. Using this we discuss the characterization of Kundt-CSI spacetimes in terms of nil-Killing vector fields.
\end{abstract}
\section{Introduction}
The scalar polynomial curvature invariants (spi's) of a pseudo-Riemannian manifold are defined as the smooth functions obtained through taking full contraction of the Riemann tensor and its covariant derivatives. They provide invariants of metrics with respect to isometries. In the class of Riemannian metrics they give a full set of invariants, in the sense that they fully characterize the orbits within this class.

Due to the presence of null-directions this no longer holds for other signatures. A particularly rich collection of examples of metrics not characterized by their spi's can be found within the class of Kundt spacetimes \cite{GenKundt,KundtSpacetimes,ExactSolutions}, which are Lorentzian manifolds that have a rank 1 null-distribution with integrable orthogonal compliment such that the distribution contains affinely geodesic, shear-free and divergence-free vector fields. A Kundt spacetime is said to be degenerate if the Riemann tensor and all its covariant derivatives are of type $II$ with respect to the given null-distribution. In \cite{CharBySpi,SCPI} the authors show that such metrics can be smoothly deformed locally such that the spi's remain fixed. In \cite{CharBySpi} they show that any three or four dimensional Lorentzian metric having such a deformation, preserving the spi's and leaving the orbit of the metric, must belong to the class of degenerate Kundt metrics.  A classification of degenerate Kundt spacetimes in terms of local coordinates was given by the authors in \cite{KundtSpacetimes}. 

We say that a Pseudo-Riemannian metric is CSI if the spi's of the metric are constant. Due to the positive definite assumption, Riemannian CSI metrics whose spi's are constant across the manifold are automatically locally homogeneous by a result in \cite{CurvatureInvariantsRiemann}. The added complexity of null-directions for Lorentzian manifolds allows for the existence of metrics whose spi's are constant on the manifold, and where the metric has no local Killing vector fields. In \cite{CSI3,CSI4} the authors show that in dimensions three and four any Lorentzian CSI metric is either locally homogeneous or Kundt, showing the importance of Kundt metrics in this setting as well.

In \cite{herviknew} the author defines nil-Killing vector fields intended to allow for the characterization of CSI metrics in analogy with the characterization of local homogeneous metrics in terms of Killing vector fields. The nil-Killing vector fields generalize Killing vector fields and simultaneously behave similarly to the affinely geodesic, shear-free and divergence-free null vector fields of a Kundt spacetime. They were studied in \cite{IDIFF} and shown to constitute a Lie algebra.

In this paper we study deformations of tensors and metrics whose deformation tensor points in the direction of  type $III$ tensors with respect to some null-distribution. We give conditions for such deformations to preserve the scalar polynomial invariants relevant to the object which is deformed. We use this to further our investigations of nil-Killing vector fields whose flows deform the metric by a tensor of type $III$, allowing for them to be  be studied in this framework. 

We find that such deformations and nil-Killing vector fields play a significant role for Kundt spacetimes:  Firstly, Kundt spacetimes can be characterized by having a null-distribution, $\lambda$, which contains local vector fields about any point that are nil-Killing with respect to the null-distribution, and  whose orthogonal complement, $\lambda^{\perp},$ is integrable . Secondly we see that the transverse metric of a Kundt metric is locally homogeneous if and only there exists a  collection of nil-Killing vector fields which pointwise span the orthogonal complement to the underlying null-distribution. Lastly, the the deformations that preserve the scalar polynomial invariants of the degenerate Kundt spacetime previously discussed, are all in the direction of type III tensors with respect to the given null-distribution.

Along the way we give a theorem for degenerate Kundt spacetimes which characterizes algebraic stability of tensors when taking covariant derivatives, in the sense that their boost-order remains unchanged. This is applied to give criteria for deformations to preserve spi's and provides us with a classification of the deformations given in \cite{CharBySpi,SCPI} as those whose deformation tensor and its covariant derivatives to all orders are of type $III$.

Lastly we shall be discussing CSI spacetimes and examining the possibility for classifying them through the use of nil-Killing vector fields. We show that such a characterization is possible in two special cases: Three dimensional degenerate Kundt spacetimes and degenerate Kundt spacetimes which admit recurrent null-vector fields. 
\section{Deformations on Lorentzian vector spaces}
In this section we shall use algebraic classification of tensors on Lorentzian vector spaces in order to study a class of deformations of type II tensors such that the invariants of the tensors stay fixed throughout the deformation. See \cite{AlgClass} and the appendix for an overview of boost-weight classification of tensors.
    
    Let us consider a triple $(V,g,\lambda)$ consisting of a vector space, a Lorentzian inner product and a one-dimensional null subspace. Using this datum we may classify the tensors on $V$ according to the boost-order. Denoting by $\mathcal{T}_{k}(V)$ the vector space of rank $k$ tensors, we get at filtration 
    \begin{equation}
        \mathcal{T}_{k,-k}(V)\subset \mathcal{T}_{k,-k+1}(V)\subset\dots \subset \mathcal{T}_{k,k}(V),
    \end{equation}
    where $\mathcal{T}_{k,s}(V)$ is the subspace given by tensors of rank $k$ and  boost order $s$ with respect to $\lambda.$ Tensors of boost order up to $0$ (resp. -1) are said to be of type $II$ (resp. type $III$).
    
    Given a nullbasis $\{k,l,m_{i}\}$ for $V$ such that $k\in \lambda$, we have the decomposition 
    
    \begin{equation}
    \mathcal{T}_{k}(V)=\bigoplus_{s=-k}^{k}\mathcal{B}_{k,s},    
    \end{equation}
    where $\mathcal{B}_{k,s}$ are the rank $k$ tensors of boost-weight $s$ with respect to $\{k,l,m_{i}\}$. The filtration and decomposition are related by
    \begin{equation}
    \mathcal{T}_{k,r}(V)=\bigoplus_{s=-k}^{r}\mathcal{B}_{k,s}.  
    \end{equation}
    
    Given tensors $S$ and $T$ the boost weight $s$ component of their tensor product is given by 
    \begin{equation}
        (S\otimes T) _{s} = \sum_{s_{1}+s_{2} =s}S_{s_1}\otimes T_{s_{2}}.\
    \end{equation}
    Furthermore, if $P\in \mathcal{T}_{2k}(V) $ is any even ranked tensor, then for any full contraction $Tr(P)$, only the boost weight zero components give a contribution, i.e.,
    \begin{equation}
        Tr(P)=Tr(P_{0}).
    \end{equation}
    Now if $S$ is any rank $k$ tensor the scalar polynomial invariants (spi's) of $S$ are given by the set of full contractions of the even ranked tensor products, i.e., \begin{equation}
        Tr(\bigotimes_{l=1}^{p}S),
    \end{equation}
    such that $pk$ is even.
    Supposing that $S$ is a tensor of type $II,$ it follows from the above considerations that \begin{equation}
        Tr(\bigotimes_{l=1}^{p}S)=Tr((\bigotimes_{l=1}^{p}S)_{0})=Tr(\bigotimes_{l=1}^{p}S_{0}),
    \end{equation}
    which shows that the spi's of $S$ only depend on the boost-weight zero component of $S.$ 
    
    Thus if $S, T$ are rank $k$ tensor of type $II$ and $III$ respectively, then $S$ and $S+T$ have identical spi's. In particular, if $S_{t}$ is a  path within the space of type $II$ tensors such that $\frac{d}{dt}S_{t}$ is a tensor of type $III$, for all $t,$ then the spi's are identical for each tensor along the curve.
    
    A partial converse follows from the proof of theorem $II.9$ and its corollary in \cite{SCPI}, namely that if $S_{t}$ is a smooth path of type $II$ tensors such that the spi's remain unchanged along the curve, then there exists a path $\phi_{t}$ in $Sim(n-2)$, the Lie group of orthogonal transformations preserving the line $\lambda,$ and a path $P_{t}$ in the space of type $III$ tensors such that \begin{equation}
        S_{t}=\phi_{t}^{*}(S_{0}) + P_{t},
    \end{equation}
    for all $t$.
    
    We have the following characterizations concerning the behaviour of linear maps with respect to algebraic type and spi's:
    
    If $(V_{i},g_{i},\lambda_{i})$, $i=1,2$, are two Lorentzian inner product spaces with given null lines, then an invertible linear map $f:V_{1}\rightarrow V_{2}$ is said to \emph{preserve algebraic type} if the pull-back $f^{*}$ preserves the boost-order of tensors taken with respect to $\lambda_{i},$ $i=1,2.$ 
    
    \begin{proposition} An invertible linear map $f$ preserves algebraic type iff. $f(\lambda_{1})\subset \lambda_{2}$ and $f(\lambda^{\perp}_{1})\subset \lambda^{\perp}_{2}$.
    \end{proposition}
    \begin{proof}
    $"\Rightarrow"$ is trivial. In order to prove $"\Leftarrow"$ we need only show that $f(\lambda_{1})\subset \lambda_{2}$ and $f(\lambda^{\perp}_{1})\subset \lambda^{\perp}_{2}$ implies that $f^{*}(\lambda_{2}^{*})\subset\lambda_{1}^*$ and $f^{*}((\lambda_{2}^{\perp})^{*})\subset(\lambda_{1}^{\perp})^*$, but this is just a general property of annihilators.
    \end{proof}
    
    \begin{proposition}\label{null-transformation}
    Let $(V_{i},g_{i},\lambda_{i})$, $i=1,2$, be two Lorentzian inner product spaces with given null lines. If $f:V_{1}\rightarrow V_{2}$ is a invertible linear map preserving algebraic type, then the following are equivalent
    \begin{enumerate}[i)]
        \item $f$ is an isometry on $\lambda_{1}^{\perp}$ and $g_{2}(fx,fz)=g_{1}(x,z)$, for all $x\in \lambda$ and $z\in V_{1}$.
        \item $f^{*}g_{2}-g_{1}$ is of type $III$ with respect to $\lambda_{1}$.
        \item If $T$ is any even ranked type $II$ tensor, then $f^{*}$ preserves full its contractions, i.e., 
        \begin{equation}
            Tr(T)=Tr(f^*T).
        \end{equation}
    \end{enumerate}
    \end{proposition}
    \begin{proof}
    $"i)\Leftrightarrow ii)"$ Trivial.\newline
    
    $"ii)\Rightarrow iii)"$
    The map $f^*:\mathcal{T}(V_{2})\rightarrow \mathcal{T}(V_{1})$ between tensor spaces is induced by the pull-back
    $$f^*:V_{2}^{*}\rightarrow V_{1}^{*}$$
    and $$f^{-1}:V_{2}\rightarrow V_{1}.$$
    
    First we wish to show that $f^{*}g_{2}^{-1}-g_{1}^{-1}$ is of type $III$. In order to do so we must show how $(f^{-1})^{*}:V_{1}^{*}\rightarrow V_{2}^{*}$ behaves with respect to taking metric duals of vectors. We let $\natural$ denote the metric dual. Suppose that for $i=1,2$, $w_{i}\in\lambda_{i}^{\perp}$, then
    \begin{equation}
        ((f^{-1})^{*}w_{1}^{\natural})(w_{2})=w_{1}^{\natural}(f^{-1}w_{2})=g_{1}(w_{1},f^{-1}w_{2})
    \end{equation}
    $$=g_{1}(f^{-1}(fw_{1}),f^{-1}w_{2})=g_{2}(fw_{1},w_{2})=(fw_{1})^{\natural}(w_{2}),$$
    and hence we have \begin{equation}(f^{-1})^{*}w^{\natural}-(fw)^{\natural}\in \lambda_{2}^{*},\end{equation}
    for all $w\in \lambda_1.$
    
    Using the same reasoning one can show that
    \begin{equation}
       (f^{-1})^{*}z^{\natural}-(fz)^{\natural}\in (\lambda_{2}^{\perp})^{*},
    \end{equation}
     and
    \begin{equation}
       (f^{-1})^{*}x^{\natural}=(fx)^{\natural},
    \end{equation}
    for all $x\in \lambda_{1},\, z\in V_1.$
    
    Thus supposing that $x\in\lambda_{1}$, $w_{1},w_{2}\in \lambda_{1}^{\perp}$ and $z\in V_1$, we have that
    \begin{equation}
        (f^{*}g_{2}^{-1})(w_{1}^{\natural},w_{2}^{\natural})=g_{2}^{-1}((f^{-1})^{*}w_{1}^{\natural},(f^{-1})^{*}w_{2}^{\natural})=g_{2}^{-1}((fw_{1})^{\natural},(fw_{2})^{\natural})
    \end{equation}
    $$=g_{2}(fw_{1},fw_{2})=g_{1}(w_{1},w_{2})=g_{1}^{-1}(w_{1}^{\natural},w_{2}^{\natural}),$$
    and in the same way
    \begin{equation}
        f^{*}g_{2}^{-1}(x^{\natural},z^{\natural})=g_{1}^{-1}(x^{\natural},z^{\natural}),
    \end{equation}
    which proves that $f^{*}g_{2}^{-1}-g_{1}^{-1}$ is of type $III$. Hence we can fined type $III$ tensors $Q,P$ such that 
    \begin{equation}
        g_{1}=f^{*}g_{2}+P,\quad g_{1}^{-1}=f^*g_{2}^{-1}+Q.
    \end{equation}
    
    Now suppose that $T$ is an even ranked type $II$ tensor, and we peform a full contraction of $T$ of form
    \begin{equation}
        Tr(T)=(g_{2}\otimes\cdots\otimes g_{2}\otimes g_{2}^{-1}\otimes\cdots \otimes g_{2}^{-1})*(T),
    \end{equation}
    where $*$ represents some full contraction between the two tensors.  The expression for the corresponding full contraction of $f^*{T}$ is given by:
    
    \begin{equation}
        Tr(f^{*}T) = (g_{1}\otimes\cdots\otimes g_{1}\otimes g_{1}^{-1}\otimes\cdots \otimes g_{1}^{-1})*(f^*T)
    \end{equation}
    $$=((f^*g_{2}+P)\otimes\cdots\otimes (f^*g_{2}+P)\otimes (f^*g_{2}^{-1}+Q)\otimes\cdots \otimes (f^*g_{2}^{-1}+Q)*(f^*T)$$
    $$\overset{(a)}{=}(f^*g_{2}\otimes\cdots\otimes f^*g_{2}\otimes f^*g_{2}^{-1}\otimes\cdots \otimes f^*g_{2}^{-1}*(f^*T)$$
    $$\overset{(b)}{=}(g_{2}\otimes\cdots\otimes g_{2}\otimes g_{2}^{-1}\otimes\cdots \otimes g_{2}^{-1})*(T)=Tr(T),$$
    Equality (a) follows since any summand containing a factor of either $P$ or $Q$ must vanish since it represents a full contraction between two tensors of type $II$ and $III$. The fact that $(b)$ holds is just a consequence of the definition of the pull-back. This proves the implication.
    
    $"iii)\Rightarrow ii)"$ Let $\{k,l,m_{i}\}$ be a null-basis for $V_{1}$ such that $k\in\lambda_{2}.$ Since $k\otimes l$ and $m_{i}\otimes m_{j}$ are type $II$ tensors, for each i,j, the assumptions of $iii)$ imply that
    \begin{equation}
        g_{1}(f^{-1}k,f^{-1}l)=Tr(f^*({k\otimes l}))=Tr({k\otimes l})=g_{2}(k,l),
    \end{equation}
    and 
    \begin{equation}
        g_{1}(f^{-1}m_{i},f^{-1}m_{j})=Tr(f^*{m_{i}\otimes m_{j}})=Tr({m_{i}\otimes m_{j}})=g_{2}(m_{i},m_{j}),
    \end{equation}
    which implies $ii).$
    \end{proof}

\section{Deformations of type II tensors on Lorentzian manifolds }\label{DefOfTensors}
Here we shall extend the discussion of the previous section and study deformations of type II tensors on Lorentzian manifolds which preserve the scalar polynomial invariants induced by the tensor. Specializing the discussion to deformations arising from one-parameter groups of diffeomorphisms, we shall find useful criteria for the spi's of the tensor to have constant values along integral curves of certain vector fields.

Now let us consider a triple $(M,g,\lambda)$ consisting of a manifold, a Lorentzian metric and a null-distribution respectively. The null distribution $\lambda$ allows us to unambiguously consider tensors of given algebraic type with respect to $\lambda.$  Whenever we say that a tensor is of a given algebraic type, it shall always be understood to be with respect to $\lambda.$

Given a tensor $T$ on $M$, the scalar polynomial invariants of $T$ are the smooth functions given by the full contractions of even-ranked tensors in the tensor algebra generated by $T.$\newline

If $S_t$ is a path of type $II$ tensors it follows from the discussion in the previous section that if $\frac{d}{dt}S_{t}$ is of type $III$, the spi's stay the same for each tensor in the curve.\newline

Now we shall investigate deformations of type II tensors induced by pull-backs by one-parameter groups of diffeomorphisms. First we have the following observation with respect to the preservation of algebraic type which follows directly from the discussion in the previous section:

\begin{observation}\label{algebraic preservation}
Let $f$ be a diffeomorphism on $(M,g,\lambda).$ The pull-back $f^{*}$ preserves algebraic type of tensors iff. $f_{*}\lambda=\lambda$ and $f_{*}\{\lambda\}^{\perp}= \{\lambda\}^{\perp}.$  In particular, if $X$ is a vector field then its flow $\phi_{t}$ preserves algebraic type, for all $t,$ iff. $[X,\lambda]\subset \lambda$ and $[X,\lambda^{\perp}]\subset \lambda^{\perp}.$
\end{observation}

Given a triple $(M,g,\lambda)$, a vector field $X$ on $M$ is said to preserve algebraic type if its flow $\phi_{t}$ preserves algebraic type, for all $t.$

We have the following result regarding the behaviour of the spi's during  deformations induced by one-parameter groups of diffeomorphisms:

\begin{proposition}\label{TypeIIIflowoftensor}
Suppose that $X$ is a vector field on $(M,g,\lambda)$ with flow $\phi_{t}$ and  $T$ is a tensor of type $II$ w.r.t. $\lambda$. If $X$ preserves algebraic type and  $\mathcal{L}_{X}T$ is of type $III$, then  $\phi_{t}^{*}T-T$ is of type $III,$ for all $t.$ In particular the spi's of $\phi_{t}^{*}T$ are identical for each $t.$
\end{proposition}
\begin{proof}
By assumption we know that
\begin{equation}
    (\mathcal{L}_{X}T)=\lim_{t\rightarrow 0}\frac{1}{t}((\phi_{t}^*T)-T)
\end{equation}
is of type III. Now suppose that $s\in \mathbb{R}$ and $p\in M.$ By observation \ref{algebraic preservation} the pull-back by $\phi_{s}$ preserves the algebraic type of tensors, and thus
\begin{equation}
    \lim_{t\rightarrow 0} \frac{1}{t}((\phi_{s+t}^{*}T)_{p}-(\phi_{s}^{*}T)_{p})=[(\phi_{s})^{*}\lim_{t\rightarrow 0}\frac{1}{t}((\phi_{t}^*T)-T)]_{p}=[(\phi_{s})^{*}(\mathcal{L}_{X}T)]_{p}
\end{equation}
is of type III. Thus $T_{t}:=\phi_{t}^{*}T$ is a smooth family of tensors of type $II$ w.r.t. $\lambda$ such that $\frac{d}{dt}T_t$ is of type $III$ w.r.t. $\lambda$, for all $t,$ and it therefore follows that $T_t-T_{0}$ is of type $III,$ for all $t.$ The discussion in the previous section shows that the spi's of $T_t$ are the same for each $t.$
\end{proof}

We recall the following definition from \cite{herviknew} and \cite{IDIFF}:
\begin{definition}
Let $(M,g)$ be a Lorentzian manifold. A vector field $X$ is said to be nil-Killing if the endomorphism corresponding to $\mathcal{L}_{X}g$ is nilpotent. 
\end{definition}

One can show that $X$ is nil-Killing iff. at each point $\mathcal{L}_{X}g$ is of type $III$ with respect to some null line at that point. Hence we have the following natural refinement of the definition:

\begin{definition}
Let $(M,g,\lambda)$ be a Lorentzian manifold with a null distribution. A vector field $X$ is said to be nil-Killing with respect to $\lambda$ if $\mathcal{L}_{X}g$ is of type $III$ with respect to $\lambda.$
\end{definition}

It was seen in \cite{IDIFF} that the collection of vector fields $X$ which are nil-Killing  w.r.t $\lambda$ such that $[X,\lambda]\subset\lambda$ is a Lie algebra, which we shall denote by \begin{equation}\mathfrak{g}_{\lambda}=\{X \text{ nil-Killing w.r.t }\lambda:[X,\lambda]\subset\lambda\}.
\end{equation}

A special class of these, the Kerr-Schild vector fields, defined as those vector fields for which $\mathcal{L}_{X}g\in \lambda^{*} \otimes \lambda^{*}$ and $[X,\lambda]\subset \lambda$ were defined and studied in \cite{Kerr}.

We have the following characterization which is a corollary to proposition \ref{TypeIIIflowoftensor}:
\begin{corollary}\label{NilFlow}
Let $(M,g,\lambda)$ be a Lorentzian manifold with a null-distribution. Suppose $X\in \mathcal{X}(M)$ with flow $\phi_t$ satisfies $[X,\lambda]\subset \lambda$. Then the following are equivalent:
\begin{enumerate}[i)]
    \item X is nil-Killing with respect to $\lambda.$
    \item $\phi_{t}^{*}g-g$ is of type $III$ w.r.t. $\lambda$.
    \item $(\phi_{t})_{*}$ is an isometry on $(\lambda^{\perp})_{p}$ and satisfies $g({(\phi_{t})}_{*}x,{(\phi_{t})}_{*}z)=g(x,z),$ for all $p\in M$, $x\in \lambda_{p}$ and $z\in T_{p}M.$
\end{enumerate}
\end{corollary}

Let us discuss how the flows of nil-Killing vector fields act with respect to the spi's of type $II$ tensors. Suppose $p\in M$ and $X\in \mathcal{X}(M)$ with flow $\phi_{t}$ is nil-Killing with respect to $\lambda$ and satisfies $[X,\lambda]\subset \lambda$. If $T$ is a tensor of type $II$, then
\begin{equation}
    (\phi_{t}^{*}T)_{p}=(\phi_{t})^{*}T_{\phi_{t}(p)},
\end{equation}
and since $\phi_{t}$ preserves algebraic type and  $\phi_{t}^*g-g$ is of type $III$ by the above corollary, it follows from proposition \ref{null-transformation} that
the corresponding spi's for $T_{\phi_{t}(p)}$ and $(\phi_{t}^{*}T)_{p}.$ have the same values for each $t.$ In particular if $X(p)=0,$ then\begin{equation}
    spi(T_{p})=spi((\phi_{t}^{*}T)_{p}),
\end{equation}
for all $t.$

Hence nil-Killing vector fields which are zero at a point $p$ allow for deformations of tensors such that the spi's at the point $p$ remain unchanged.

Now we give criteria for the spi's of a type $II$ tensor to stay constant along the integral curves of a vector field. 

\begin{proposition}\label{SpiPreservationOfTensor}
Let $(M,g,\lambda)$ be a Lorentzian manifold with a null-distribution and $p\in M$. Suppose that $T$ is a tensor and $X$ is a vector field such that $T_{p}$ is of type $II$ and $(\mathcal{L}_{X}g)_{p}$, $(\mathcal{L}_{X}T)_{p}$ are of type  $III$ with respect to $\lambda_{p}$. Then \begin{equation}
    X_{p}(\mathcal{I})=0,
\end{equation}
for all spi's $\mathcal{I}$ of $T.$
\end{proposition}
\begin{proof} Firstly we see that  
\begin{equation}
    0=\mathcal{L}_{X}(g*g^{-1})=(\mathcal{L}_{X}g)*g^{-1}+g*(\mathcal{L}_{X}g^{-1}),
\end{equation}
and hence
\begin{equation}
    (\mathcal{L}_{X}g^{-1})_{p}=g_{p}^{-1}*(\mathcal{L}_{X}g)_{p}*g_{p}^{-1},
\end{equation}
which implies that $(\mathcal{L}_{X}g^{-1})_{p}$ is of type $III.$

Any spi, $I,$ of $T$ can be written as some full contraction
\begin{equation}
   \mathcal{I}=(g\otimes \cdots \otimes g\otimes g^{-1}\otimes \cdots \otimes g^{-1})*(T\otimes \cdots \otimes T).
\end{equation}
Therefore using the Leibniz rule for the Lie derivative  $X_{p}(\mathcal{I})=(\mathcal{L}_{X}(\mathcal{I}))_{p}$ on the expression above, we see that each resulting summand is a full contraction between tensors of type $II$ and $III$ which therefore must vanish. Hence
\begin{equation}
    X_{p}(\mathcal{I})=0,
\end{equation}
proving the proposition.
\end{proof}

\begin{corollary}
Let $(M,g,\lambda)$ be a Lorentzian manifold with a null distribution and suppose that $M$ has a transitive collection $\{X_{i}\}_{i\in I}$ of vector fields that are nil-Killing w.r.t. $\lambda.$ If $T$ is a tensor of type $II$ such that  $\mathcal{L}_{X_{i}}T$ is of type $III$, for all $i\in I$, then the spi's of $T$ are constant on $M.$
\end{corollary}

\section{Lie algebras induced by algebraic type}\label{Lie algebra}

In the previous sections we saw the utility of studying vector fields preserving algebraic type and whose Lie derivative brings the metric $g$ and a given type $II$ tensor $T$ to algebraic type $III$. Namely this gives a useful criterion for when the spi's of $T$ are constant along the integral curves of $X.$ 

In this section we shall generalize this discussion in two ways. We shall let the algebraic type be given by any choice of boost-order, and instead of a single tensor we shall consider collections of tensors. Our main result shall be on the construction of Lie algebras specified by this datum. We follow up with a discussion of interesting cases.

Let $(M,g,\lambda)$ be a Lorentzian manifold with a null-distribution $\lambda$. Now suppose we are given a collection of tensors $\{T_{i}\}_{i\in I}$ and an integer $s.$ We have the following generalization of proposition 2.5 in \cite{IDIFF}:
\newline

\begin{proposition}\label{AlgebraicStructureLiealgebra}
The collection of vector fields $X\in \mathcal{X}(M)$ satisfying \begin{enumerate}[i)] \item$ [X,\lambda]\subset \lambda,$ 
\item $[X,\lambda^{\perp}]\subset \lambda^{\perp},$
\item $\mathcal{L}_{X}T_{i}$ has boost-order $\leq s$ w.r.t $\lambda,$ for all $i\in I,$
\end{enumerate}
is a Lie algebra which we denote by $\mathfrak{g}_{\{T_{i}\}_{i\in I}}^s$.
\end{proposition}
\begin{proof}
Suppose $X$ and $Y$ are vector fields lying in this collection. First let us verify that $[X,Y]$ satisfies $i),ii).$ This is easy, since by the Bianchi identity we have
\begin{equation}
    [[X,Y],k]=-[[Y,k],X]-[[k,X],Y]\in \lambda
\end{equation}
since $[Y,k],[k,X]\in \lambda,$ and thus $[[X,Y],\lambda]\subset \lambda.$ For the exact same reason we also have $[[X,Y],\lambda^{\perp}]\subset\lambda^{\perp}.$

Now suppose that $i\in I$, then 
\begin{equation}\label{Bianchi}
    \mathcal{L}_{[X,Y]}T_{i}=\mathcal{L}_{X}(\mathcal{L}_{Y}T_{i})-\mathcal{L}_{Y}(\mathcal{L}_{X}T_{i})
\end{equation}
By assumption $\mathcal{L}_{X}T_{i}$ and $\mathcal{L}_{Y}T_{i}$ are of boost-order $\leq s$. Furthermore observation \ref{algebraic preservation} implies that the flows of both vector fields preserve algebraic type with respect to $\lambda.$ Hence, letting $\phi_{t}$ denote the flow of $X$ we have
\begin{equation}
    \mathcal{L}_{X}(\mathcal{L}_{Y}T_{i})=\lim_{t\rightarrow 0}\frac{1}{t}(\phi_{t}^{*}(\mathcal{L}_{Y}T_{i})-\mathcal{L}_{Y}T_{i}),
\end{equation}
has boost order $\leq s,$ since it is the limit of tensors of boost-order $\leq s.$ For the same reason $\mathcal{L}_{Y}(\mathcal{L}_{X}T_{i})$ has boost-order $\leq s.$ Thus \eqref{Bianchi} shows that $[X,Y]$ is in the collection. Hence the collection of such tensors give a Lie algebra.
\end{proof}

We proceed by giving a few examples of such Lie algebras that will be important to us. In the following let $(M,g,\lambda)$ be a Lorentzian manifold with a null-distribution
\begin{enumerate}
\item The Lie algebra of nil-Killing vector fields which preserve algebraic structure, as descibed in section \ref{DefOfTensors}, is obtained by letting the collection of tensors be given by $\{g\}$ and setting $s=-1.$
\item Taking the collection to be $\{g\}\cup\{T_{i}\}_{i\in I}$ where $T_{i}$ is of type $II$ w.r.t. $\lambda,$ for all $i\in I,$ and setting $s=-1,$ then by proposition \ref{SpiPreservationOfTensor} the vector fields in the resulting Lie algebra preserve the spi's generated by the collection $\{T_{i}\}_{i\in I}$.
\item In the particular case that the curvature and all its covariant derivative are of type $II$ w.r.t $\lambda$ then $s=-1$ together with the collection $\{g,\nabla^{m}Rm\}$  induces a Lie algebra whose vector fields preserve the scalar curvature invariants of the metric.
\end{enumerate}

Now consider again a collection $\{T_i\}_{i\in I}$ and an integer $s.$ Let
\begin{equation}
\mathcal{G}^{s}_{\{T_{i}\}_{i\in I}}=\{\phi\in \text{Diff}(M):\phi_*(\lambda)\subset \lambda,\,\phi_*(\lambda^\perp)\subset\lambda^\perp,\,\phi^*(T_{i})-T_i \text{ has b.o. } \leq s, \forall i\in I\}.
\end{equation}
Let us show that $\mathcal{G}^{s}_{\{T_{i}\}_{i\in I}}$ is a group. Suppose that $\phi,\psi\in \mathcal{G}$, the obviously their composition $\phi\psi$ preserves the distribution $\lambda$ and $\lambda^\perp$ since they indidually have this property. Suppose that $i\in I,$ then
$\phi^*(T_{i})-T_{i}$ is of boost-order $\leq s.$ Since $\psi$ preserves the distributions it follows that $\psi^{*}(\phi^*(T_{i})-T_{i})=(\phi\psi)^*(T_{i})-\psi^*(T_{i})$ is of boost-order $\leq  s.$ Moreover
\begin{equation}
    (\phi\psi)^*(T_{i})-\psi^*(T_{i})=[(\phi\psi)^*(T_{i})-T_{i}]-[\psi^*(T_{i})-T_{i}],
\end{equation}
from which it follows that $(\phi\psi)^*(T_{i})-T_{i}$ has boost-order $\leq s$, and so $\phi\psi \in \mathcal{G}^{s}_{\{T_{i}\}_{i\in I}}. $ Now let us show closure under taking inverses. If $\psi\in \mathcal{G}^{s}_{\{T_{i}\}_{i\in I}}$, then clearly $\psi^{-1}$ preserves algebraic structure. By assumtion, if $i\in I$ then
\begin{equation}
    \psi^{*}((\psi^{-1})^*T_{i}-T_{i})=T_{i}-\psi^{*}T_{i},
\end{equation}
is of boost-order $\leq s.$ Since $\psi^{-1}$ preserves algebraic structure, we see that 
\begin{equation}
    (\psi^{-1})^*[\psi^{*}((\psi^{-1})^*T_{i}-T_{i})] =(\psi^{-1})^*T_{i}-T_{i}
\end{equation}
must also be of boost-order $\leq s.$ This shows that $\psi^{-1}\in \mathcal{G}^{s}_{\{T_{i}\}_{i\in I}}$ showing that the set is a group.

 Given an integer $s$ and a collection of tensors $\{T_{i}\}_{i\in I}$, then following the steps in section \ref{DefOfTensors}, we see that a vector field $X$ with flow $\phi_t$ is in the the Lie algebra $\mathfrak{g}_{\{T_{i}\}_{i\in I}}^s$ iff. $\phi_{t}\in \mathcal{G}^{s}_{\{T_{i}\}_{i\in I}},$ for all $t.$

Thus we can think of $\mathfrak{g}_{\{T_{i}\}_{i\in I}}^s$ as the Lie algebra of the group $\mathcal{G}^{s}_{\{T_{i}\}_{i\in I}}.$
\section{Local description of nil-Killing vector fields}

In this section we give a coordinate description of Nil-Killing vector fields on a Lorentzian manifold admitting a twist-free, geodesic null-congruence.

Let $(M,g)$ be a Lorentzian space-time with a twist-free and geodesic null vector field $k$. By rescaling $k$ if necessary we can find local coordinates $(u,v,x^{i})$ with $i=1\dots n-2$ such that $k=\frac{\partial}{\partial v}$ with metric dual $k^\natural=du$ and
\begin{equation}
    g= 2du(dv + Hdu +W_{i}dx^{i}) + \tilde{g}_{ij}(u,v,x^{k})dx^{i}dx^{j}.
\end{equation}

Let $\lambda$ be the null-distribution, which at each point is defined to be the span of $k.$
Now suppose that $X=\mathcal{X}(M)$ is nil-Killing with respect to $\lambda$, then since $\mathcal{L}_{X}g(k,k)=0,$ we see that $[X,\lambda]\in \{\lambda\}^{\perp},$ and thus $X$ expressed in coordinates is of the form
\begin{equation}
    X:=A(u,x^{k})\frac{\partial}{\partial u} + B(u,v,x^{k})\frac{\partial}{\partial v} + C^{i}(u,v,x^{k})\frac{\partial}{\partial x^{i}},
\end{equation}
where $A$ does not depend on $v.$ 

Now calculating the Lie derivative of the metric with respect to $X$ gives us

\begin{equation}
    \mathcal{L}_{X}g=(\mathcal{L}_{X}g)_{II} + (\mathcal{L}_{X}g)_{III}+(\mathcal{L}_{X}g)_{r},
\end{equation}
where
\begin{equation}\
    (\mathcal{L}_{X}g)_{II} := 2(X(H) + 2H\frac{\partial A}{\partial u} + \frac{\partial B}{\partial u} + W_{i}\frac{\partial C^{i} }{\partial u})dudu,
\end{equation}
\begin{equation}\label{coordinateordertwonil}
(\mathcal{L}_{X}g)_{III}:=  2(X(W_{i}) + 2H\frac{\partial A }{\partial x^{i}}+\frac{\partial B}{\partial x^{i}} + W_{j}\frac{\partial C^{j}}{\partial x^{i}} +W_{i}\frac{\partial A}{\partial u} + \tilde{g}_{ij}\frac{\partial C^{j}}{\partial u})dudx^{i},
\end{equation}
and
\begin{equation}\label{CoordinateNilkilling}
(\mathcal{L}_{X}g)_{r}:=2(\frac{\partial A}{\partial u} + \frac{\partial B}{\partial v}+ W_{i}\frac{\partial C^{i}}{\partial v})dudv+2(\frac{\partial A}{\partial x^{i}}+ \tilde{g}_{ij}\frac{\partial C^{j}}{\partial v})dx^{i}dv
\end{equation}
$$+(C^{k}\frac{\partial \tilde{g}_{ij}}{\partial x^{k}} + \tilde{g}_{jk}\frac{\partial C^{k}}{\partial x^{i}}+\tilde{g}_{ik}\frac{\partial C^{k}}{\partial x^{j}} + A\frac{\partial \tilde{g}_{ij}}{\partial u}+B\frac{\partial \tilde{g}_{ij}}{\partial v} + W_{i}\frac{\partial A}{\partial x^{j}}+W_{j}\frac{\partial A}{\partial x^{i}})dx^{i}dx^{j}.$$

Hence we see that a vector field $$X:=A\frac{\partial}{\partial u} + B\frac{\partial}{\partial v} + C^{i}\frac{\partial}{\partial x^{i}},$$ is Nil-Killing [of order two] w.r.t. $\lambda$ iff.
$$\frac{\partial A}{\partial v}$$ and
$$(\mathcal{L}_{X}g)_{r}=[(\mathcal{L}_{X}g)_{III}]=0.$$ 
\newline


Let us proceed by classifying the collection \begin{equation}\hat{\mathcal{N}}_{\lambda}=\{Y\in \mathcal{X}(M):Y\text{ is Nil-Killing w.r.t. } \lambda,[Y,\lambda]\subset \lambda\}\end{equation} in terms of coordinates. 

Suppose that $X$ is Nil-Killing vector fields $X$ such that $[X,\lambda]\subset \lambda$. This is equivalent to  $(\mathcal{L}_{X}g)_{r}=0$ and $\frac{\partial A}{\partial v}=\frac{\partial C^{i}}{\partial v}=0 $ for $i=1\dots n-2$. By inspection these equations are equivalent to
\begin{equation}
    \frac{\partial A}{\partial v}=\frac{\partial C^{i}}{\partial v}=0,
\end{equation}
\begin{equation}
    \frac{\partial B}{\partial v}=-\frac{\partial A}{\partial u},
\end{equation}
\begin{equation}
    \frac{\partial A}{\partial x^{i}}=0,
\end{equation}
and
\begin{equation}\label{Ceq}
    C^{k}\frac{\partial \tilde{g}_{ij}}{\partial x^{k}} + \tilde{g}_{jk}\frac{\partial C^{k}}{\partial x^{i}}+\tilde{g}_{ik}\frac{\partial C^{k}}{\partial x^{j}} + A\frac{\partial \tilde{g}_{ij}}{\partial u}+B\frac{\partial \tilde{g}_{ij}}{\partial v}=0.
\end{equation}
In summary we have the following characterization for space-times with a twist-free and geodesic null vector field:
\begin{proposition}\label{NilKillinInCoordinates}
Given a metric
\begin{equation}
    g:=2du(dv+Hdu+W_{i}dx^{i})+\tilde{g}_{ij}(u,v,x^{k})dx^{i}dx^{j}
\end{equation}with a twist-free and geodesic null-vector field $\frac{\partial}{\partial v}$. Let $\lambda$ be the null distribution spanned by $\frac{\partial  }{\partial v}.$ The collection $\hat{\mathcal{N}}_{\lambda}$ of vector fields $X$ that are nil-Killing with respect to $\lambda$ and that satisfy $[X,\lambda]\subset \lambda$ is given by
    \begin{equation}
        X=A(u)\frac{\partial}{\partial u}+(-v\frac{\partial A}{\partial u}(u)+B(u,x^{k}))\frac{\partial}{\partial v}+C^{i}(u,x^{k})\frac{\partial}{\partial x^{i}},
    \end{equation}
    such that the functions  $A(u)$, $C^{i}(u,x^{k})$  and $B(u,x^{k})$ satisfy \eqref{Ceq}. 
    \end{proposition}
    \section{Characterizations for Kundt spacetimes}
    In this section we shall discuss Kundt spacetimes \cite{KundtSpacetimes,GenKundt,ExactSolutions} and relate them to the concept of nil-Killing vector fields. We give a characterization of Kundt spacetimes having a locally homogeneous transverse metric. Following this we find a helpful classification of degenerate Kundt spacetimes which follows readily from results in \cite{KundtSpacetimes}. 
    
    We start by introducing some definitions. Let $(M,g,\lambda)$ be a Lorentzian manifold along with a rank $1$ null-distribution $\lambda.$ We shall say that the triple is \emph{Kundt} if the distribution $\lambda^{\perp}$ is integrable and at each point $p\in M$ we can find a local vector field $k\in \lambda$ which is shear-free, affinely geodesic and divergence-free, i.e.,
    \begin{equation}\label{nil-Killing}
        \nabla_{(a}k_{b)}\nabla^{a}k^{b}=0,\quad\nabla_{k}k=0\,\text{ and }\, \nabla_{a}k^{a}=0, 
    \end{equation}
    respectively. If $(M,g,\lambda)$ is Kundt we shall say that a local vector field $k\in \lambda$ is Kundt if it satisfies the conditions \eqref{nil-Killing}. We have the following characterization:
    
    \begin{proposition}\label{KundtChar}
Let $(M,g)$ be a Lorentzian manifold and $k$ a null vector field with flow $\phi_{t}$. The following are equivalent:
\begin{enumerate}[i)]
    \item $k$ is affinely geodesic, shear-free and divergence-free.
    \item $k$ is nil-Killing with respect to $k$.
    \item $(\phi_{t})_{*}g-g$ is of type $III$ with respect to $k,$ for all $t.$
    \item $(\phi_{t})_{*}$ is an isometry on $(\lambda^{\perp})_{p}$ and satisfies $g({(\phi_{t})}_{*}k,{(\phi_{t})}_{*}z)=g(k,z),$ for all $p\in M$ and $z\in T_{p}M.$
\end{enumerate}
\end{proposition}
\begin{proof}
$i)\Leftrightarrow ii)$ The divergence, sheer and acceleration of $k$ is given by
\begin{equation}
    \nabla^{a}k_{a}=\tensor{(\mathcal{L}_{k}g)}{^{a}_{a}},\quad \nabla_{(a}k_{b)}\nabla^{a}k^{b}=(\mathcal{L}_{k}g)_{ab}(\mathcal{L}_{k}g)^{ab}
\end{equation}
and
\begin{equation}
    k^{a}\nabla_{a}k_{b}=k^{a}(\mathcal{L}_{k}g)_{ab}
\end{equation}
respectively. 

If $k$ is nil-Killing with respect to $k,$ then since $\mathcal{L}_{k}g$ is nilpotent with respect to $k$, we see by the above expressions that $k$ must be divergence-free, sheer-free and affinely geodesic.

Conversely, if $k$ is divergence-free, sheerfree and affinely geodesic. Then \begin{equation}k^{a}(\mathcal{L}_{k}g)_{ab}=0,\end{equation} therefore letting $\{k,l,m^{1},\dots,m^{n-2}\}$ be some null frame completing $k$  we can write 
\begin{equation}
    \mathcal{L}_{k}g_{ab}=k_{a}W_{b} + W_{a}k_{b} +\sum_{ij}a_{ij}{m^{i}}_{a}{m^{j}}_{b}
\end{equation}
for some symmetric $(n-2)\times (n-2)$ matrix of function $a_{ij}$. Furthermore we see that
\begin{equation}
    (\mathcal{L}_{k}g)_{ab}(\mathcal{L}_{k}g)^{ab}=\sum_{ij}(a_{ij})^2,
\end{equation}
which by the sheer-free condition implies that $a_{ij}=0.$ Hence $k$ is nil-Killing with respect to $k$.

$ii)\Leftrightarrow iii)$ and $iii)\Leftrightarrow iv).$
 These are just a special cases of corollary \ref{NilFlow}.
\end{proof}

\begin{proposition}\label{TwistLemma}
Suppose that $(M,g,\lambda)$ is a Lorentzian manifold with a null distribution $\lambda$ such that about each point $p\in M$ there is a local vector field $k\in \lambda$ which is nil-Killing w.r.t. $\lambda$. Then the following are equivalent:
\begin{enumerate}[i)]
\item $\lambda^{\perp}$ is an integrable distribution.
\item $\nabla_{\lambda^{\perp}} \lambda\in \lambda$.
\item $\nabla_{\lambda^{\perp}}\lambda^{\perp}\in \lambda^{\perp}$.
\end{enumerate}
 \end{proposition}
\begin{proof}
$i)\Leftrightarrow ii)$ Let $\{k,l,m^1,\dots,m^{(n-2)}\}$ be a null frame on an open set $U$ such that $k\in \lambda$ and $k$ is nil-Killing with respect to $\lambda.$ Then $\nabla_{(a}k_{b)}=k_{a}U_{b}+{U}_{a}k_{b},$ for some $U\in \lambda^{\perp}$. From this it follows that there exists $P,Q\in \{k\}^{\perp}$ and functions $\alpha_{ij}$ with $1\leq i<j\leq n-2$ such that  $\nabla_{a}k_{b}=k_{a}P_{b}+Q_{a}k_{b} + \sum_{i<j}\alpha_{ij}m^{i}_{[a} m^{j}_{b]}. $ Clearly then 
$$k_{[c}\nabla_{a}k_{b]}=\sum_{i<j}\alpha_{ij}k_{[c}m^{i}_{a}m^{j}_{b]}.$$ From which we see that $\lambda$ is twistfree on $U$ iff. $\alpha_{ij}=0$, for all $i,j.$ It is clear that $\alpha_{ij}$ is zero, for $i<j$ iff. $\nabla_{\lambda^{\perp}}\lambda\subset \lambda.$

$ii)\Leftrightarrow iii)$ Suppose $W,W^{'}\in \{k\}^{\perp}$. Then
$k^{b}W^{a}\nabla_{a}{W^{\prime}}_{b}=-W^{a}{W^{\prime}}_{b}\nabla_{a} k^{b}$, from which the equivalence follows.
\end{proof}
\begin{proposition}\label{Kundttransverse}
Suppose that $(M,g,\lambda)$ is a Kundt spacetime. If $(u,v,x^{k})$ are coordinates on a neighborhood $U\subset M$ such that $k:=\frac{\partial}{\partial v}\in \lambda$ and $k^{\natural}=du,$ then writing
\begin{equation}
    g=2du(dv+Hdu + W_{i}dx^{i})+\tilde{g}_{ij}(u,x^{k})dx^{i}dx^{j},
\end{equation}
for some smooth functions $H$ and $W_{i}$, $i=1,\dots n-2,$ the following are equivalent:
\begin{enumerate}[i)]
\item For each point $p\in U$ there exists $(n-2)$  local space-like vector fields about $p$ that are nil-Killing w.r.t. $\lambda$, belong to $\lambda^{\perp}$ and are linearly independent at $p.$
\item  For any fixed fixed $u,$ the transverse metric $g^{\perp}_{ij}(u,x^{k})dx^{i}dx^{j}$ is locally homogeneous.
\end{enumerate}
Moreover the transverse metric is independent of $u,$ $\frac{\partial}{\partial u}\tilde{g}_{ij}(u,x^{k})=0,$ if and only if $\frac{\partial}{\partial u}$ is nil-Killing w.r.t. $\lambda.$
\end{proposition}
\begin{proof}

$"\Rightarrow"$
If $p\in U$, let $X_{1},\dots X_{n-2}$ satisfying the assumptions in $i).$ If $l\in \{1,\dots n-2\}$, then since $X_{l}$ belongs to $\lambda^{\perp}$ and is nil-Killing with respect to $\lambda$ we know that,
$$\mathcal{L}_{X_{l}}g(k,W)=-g(k,[X_{l},W])-g([X_{l},k],W)=-g([X_{l},k],W)=0,$$
for all $W\in \lambda^{\perp}$. Thus $[X_{l},k]\in \lambda$ which by proposition \ref{NilKillinInCoordinates}, implies that there exists functions $B_{l}(u,x^{k})$ and $C_{l}^{s}(u,x^{k})$, for $l,s=1,\dots n-2,$ satisfying
\begin{equation}
    C_{l}^{s}\frac{\partial \tilde{g}_{ij}}{\partial x^{s}} + \tilde{g}_{js}\frac{\partial C_{l}^{s}}{\partial x^{i}}+\tilde{g}_{is}\frac{\partial C_{l}^{s}}{\partial x^{j}}.
\end{equation}
where
\begin{equation}X_{l}=B_{l}(u,x^{k})\frac{\partial }{\partial v}+C_{l}^{s}\frac{\partial}{\partial x^{s}}.\end{equation}
Hence for each $u$ the collection
\begin{equation}
\tilde{X}_{u,l}(x^{k})= C_{l}^{s}(u,x^{k})\frac{\partial}{\partial x^{s}}
\end{equation}
for $l=1,\dots, n-2,$ constitutes a local collection of Killing vector fields for the transverse metric $\tilde{g}_{ij}(u,x^{k})dx^{i}dx^{j},$ which are linearly independent at the point $p.$ Since the point $p$ was arbitrarily chosen this implies that the transverse metric
$\tilde{g}_{ij}(u,x^{k})dx^{i}dx^{j}$ must be locally homogeneous, for each $u.$

$"\Leftarrow"$
Suppose that $\tilde{g}_{ij}(u,x^{k})dx^{i}dx^{j}$ is locally homogeneous for each $u$. Given $u_{0}$ and $x_{0}^{k}$, then there exists a smoothly parametrized collection $X_{u,l}=C_{u,l}^{s}(x^{k})\frac{\partial}{\partial x^{s}}$, $i=1,\dots, n-2$, defined for $u$ in some interval about $u_{0}$ and $x^{k}$ in some fixed neighborhood about $x_{0}^{k},$  such that for each $u$, $\{X_{u,l}\}_{l=1,\dots n-2}$ is a pointwise linearly independent collection of local Killing vector fields for $\tilde{g}_{ij}(u,x^{k})dx^{i}dx^{j}$. Hence letting $C_{l}^{s}(u,x^{k}):=C_{u,l}^{s}(x^{k})$ for each $u$, the collection
$X_{l}=C_{l}^{s}(u,x^{k})\frac{\partial }{\partial x^{s}}$ satisfies $[X_{l},k]=0$ and equation \eqref{Ceq} and therefore by proposition \ref{NilKillinInCoordinates} the vector fields $X_{l}$ are nil-Killing with respect to $\lambda$, for all $l,$ and satisfy the assumptions in $i).$

Lastly
\begin{equation}
    \mathcal{L}_{\frac{\partial}{\partial u}}g=2du((\frac{\partial}{\partial u}H)du+(\frac{\partial}{\partial u}W_{i})dx^{i}) + (\frac{\partial}{\partial u}\tilde{g})dx^{i}dx^{j},
\end{equation}
from which we observe that $\frac{\partial}{\partial u}$ is nil-Killing w.r.t. $\lambda$ iff. $\frac{\partial}{\partial u}\tilde{g}=0.$

\end{proof}

Following \cite{KundtSpacetimes} we say that a Lorentzian manifold with a rank $1$ distribution, $(M,g,\lambda)$, is \emph{degenerate Kundt} if it is Kundt and the Riemannian curvature tensor and all its covariant derivatives, $\nabla^{m}Rm$ for $m\geq 0,$ are of type $II$ w.r.t. $\lambda.$ In this case we shall simply say that $(M,g,\lambda)$ is degenerate Kundt. With a little work the following characterization follows readily from \cite{KundtSpacetimes}.

\begin{proposition}\label{DegkundtChar}
Let $(M,g,\lambda)$ be Kundt and suppose that $k\in \lambda$ is a Kundt vector field defined on some open set $U\subset M.$
Then
\begin{enumerate}[i)]
    \item The curvature $R$ is of type $II$ w.r.t. $\lambda$ on $U$ iff. $(\mathcal{L}_{k})^2g$ has boost-order $\leq -2$.
    \item $(U,g,\lambda)$ is degenerate Kundt iff. \begin{equation}(\mathcal{L}_{k})^2g\end{equation} has boost-order $\leq -2,$ and \begin{equation}(\mathcal{L}_{k})^3g=0.\end{equation}
\end{enumerate}
\end{proposition}
\begin{proof}
Let $(u,v,x^{i})$ be coordinates on an open set $V\subset M$ such that $\frac{\partial}{\partial v}$ is a Kundt vector field belonging to $\lambda$ and
$$g=2du(dv+H(u,v,x,y)du + W_{i}(u,v,x^{k})dx^{i})+{g^{\perp}}_{ij}(u,x^{k})dx^{i}dx^{j}.$$ Then \begin{equation}
    (\mathcal{L}_{\partial_v})^2g=2du(H_{,vv}du + W_{i,vv}dx^{i})
\end{equation}
and
\begin{equation}
    (\mathcal{L}_{\partial_v})^3g=2du(H_{,vvv}du+W_{i,vvv}dx^{i})
\end{equation}
By \cite{KundtSpacetimes} we know that the Riemannian curvature is of type $II$ iff. $W_{i,vv}=0$, for all $i,$ which we see is satisfies iff. $(\mathcal{L}_{\partial_{v}})^2g$ is of boost-order $\leq -2$ with respect to $\lambda$. Furthermore in \cite{KundtSpacetimes} $(M,g,\lambda)$ is degenerate Kundt on $V$ iff. $H_{,vvv}=0$ and $W_{i,vv}=0$, for all $i$ which is satisfied iff. $(\mathcal{L}_{\partial_v})^2g$ is of boost-order $\leq -2$ with respect to $\lambda$ and $(\mathcal{L}_{\partial_{v}})^3g=0.$

If $k$ is any Kundt vector field on $V$ belonging to $\lambda,$ we can write $k=f\partial_v$ for some smooth function $f$ such that $k(f)=0.$ Furthermore
\begin{equation}
    \mathcal{L}_{k}g=\mathcal{L}_{f\partial_v}g=df\otimes_{s}du+f\mathcal{L}_{\partial_v}g,
\end{equation}

\begin{multline}
    (\mathcal{L}_{k})^2g=\mathcal{L}_{f\partial_v}\mathcal{L}_{k}g=df\otimes_s(\partial_v*\mathcal{L}_{k}g)+f\mathcal{L}_{\partial_v}\mathcal{L}_{k}g\\=f\mathcal{L}_{\partial_v}(df\otimes_{s}du+f\mathcal{L}_{\partial_v}g)=f^2(\mathcal{L}_{\partial_v})^2g
\end{multline}
and lastly
\begin{multline}
    (\mathcal{L}_{k})^3g=\mathcal{L}_{k}f^2((\mathcal{L}_{\partial_v})^2g)=f^2\mathcal{L}_{k}((\mathcal{L}_{\partial_v})^2g))\\=f^2(df\otimes(\partial_v*(\mathcal{L}_{\partial_v})^2g))+f(\mathcal{L}_{\partial_{v}})^3g)=f^3(\mathcal{L}_{\partial_{v}})^3g,
\end{multline}
where $*$ denotes some contraction. Thus $(\mathcal{L}_{k})^2g$ is of boost-order $-2$ iff the same is true with $\partial_v$ and  $(\mathcal{L}_{k})^3g=0$ iff. $(\mathcal{L}_{\partial_v})^3g=0$. This finishes the proof.
\end{proof}
Hence $(M,g,\lambda)$ is degenerate Kundt iff. we can find a cover $\{U_{\alpha}\}_{\alpha \in I}$ of $M$ along with kundt vector fields $k_{\alpha}$ on $U_{\alpha}$ belonging to $\lambda$ such that $(\mathcal{L}_{k_\alpha})^2g$ is of boost-order $\leq -2$ w.r.t $\lambda$ and $(\mathcal{L}_{k_\alpha})^3g=0$, for all $\alpha\in I$.
\newline

\section{Algebraic stability for degenerate Kundt spacetimes}

In this section we characterize tensors on degenerate Kundt manifolds which are algebraically stable in the sense that taking arbitrary covariant derivatives leaves the algebraic type unchanged.

We shall start by discussing the behaviour of a class of frames.
\begin{characterization}\label{DegKundtFrame}
Given a Lorentzian manifold $(M,g)$, then a null-frame  $\{k,l,m_{i}\}$ satisfies
\begin{enumerate}[i)]
    \item $\mathcal{L}_{m_{i}}m_{j}\in {k}^{\perp},$ for all $i,j,$
    \item $\mathcal{L}_{k}l\in {k}^{\perp},\quad (\mathcal{L}_{k})^2l\in \mathbb{R}k, \quad (\mathcal{L}_{k})^3l=0,$
    \item $\mathcal{L}_{k}m_{i}\in \mathbb{R}k, \quad (\mathcal{L}_{k})^{2}m_{i}=0,$
\end{enumerate}
 if and only if for the connection coefficients, $\Gamma_{\alpha\beta\gamma}=g(e_{\alpha},\nabla_{e_{\gamma}}e_{\beta})$, the following holds:
\begin{enumerate}
    \item $\Gamma_{\alpha\beta\gamma}$ vanishes whenever $\alpha\beta\gamma$ is a strictly positive boost-weight index.
    \item Given an integer $s\geq 0,$ then $k^{(s+1)}\Gamma_{\alpha\beta\gamma}=0$ when $\alpha\beta\gamma$ is an index of boost-weight $-s.$
\end{enumerate}
\end{characterization}

An important property of such frames is that if $s$ is an integer and $e_{i_1}\cdots e_{i_{r}}$ is a tensor product of frame elements having boost-weight $s,$ then \begin{equation}(\mathcal{L}_{k})^{j}(e_{i_1}\cdots e_{i_{r}})\end{equation} has boost-order $\leq s-j,$  for all integers $j\geq 0.$

\begin{definition}
If $(M,g,\lambda)$ is a Lorentzian manifold with a null-distribution $\lambda,$ a  local frame $\{k,l,m_{i}\}$ satisfying $k\in \lambda$ and $i)-iii)$ in characterization \ref{DegKundtFrame} is said to be a degenerate Kundt frame.
\end{definition}

By application of proposition \ref{DegkundtChar} it can be shown that $(M,g,\lambda)$ is degenerate Kundt iff. about each point of $M$ there exists a degenerate Kundt frame for $(M,g,\lambda)$. This explains the chosen name for such frames.

The following lemma will be useful in the characterization of algebraic stability.
\begin{lemma}\label{componentlemma}
Suppose that $(M,g,\lambda)$ is a degenerate Kundt spacetime and let $\{k,l,m_{i}\}$ be a degenerate Kundt frame on some open set $U\subset M$. If $T$ is a rank $r$ tensor on $U$ then $
(\mathcal{L}_{k})^{j}T$ is of boost-order $\leq s-j$, for all $j\geq 0$ iff. $T$ is of boost-order $\leq s$ and each boost-weight $s+1-j$ component, $A$, of $T$ satisfies $(k)^{j}A=0$, for all $0\leq j.$
\end{lemma}

$"\Rightarrow"$ If $j\geq 0$ is an integer then suppose that $\beta_{1}\cdots\beta_{r}$ is an index of b.w. $-j+s+1.$ Then $(k)^j(T_{\beta_{1}\cdots \beta_{r}})$ can be written as a sum of terms of the form
\begin{equation}
    ((\mathcal{L}_{k})^{q_1}T)_{a_1\cdots a_{r}}((\mathcal{L}_{k})^{q_2}(e_{\beta_1}\cdots e_{\beta_r}))^{a_{1}\cdots a_{r}},
\end{equation}
where $q_1+q_2=j.$ Such a term is a full contraction of a tensor of  boost-order $$\leq (-q_1 +s)+ (j-s-1-q_2) = -q_1 -q_2-1 +j=-1,$$ which must therefore vanish. This proves the implication.

$"\Leftarrow"$ We prove by induction on $j$ that $(\mathcal{L}_{k})^{j}T$ is of boost-order $\leq s-j$, for all $j\geq 0.$ Since $T$ is of boost-order $\leq s,$ this is trivially true for $j=0.$ We prove it also for $j=1.$ If $\alpha_{1}\cdots\alpha_{r}$ is a b.w. $s$ index, then
\begin{equation}
    k(T_{\alpha_{1}\cdots \alpha_{r}})=(\mathcal{L}_{k}T)_{\alpha_{1}\cdots \alpha_{r}}+T_{a_1\cdots a_{r}}\mathcal{L}_{k}(e_{\alpha_{1}}^{a_{1}}\cdots e_{\alpha_{r}}^{a_{r}}) = (\mathcal{L}_{k}T)_{\alpha_{1}\cdots \alpha_{r}},
\end{equation}
where the last equality holds since $\mathcal{L}_{k}(e_{\alpha_{1}}^{a_{1}}\cdots e_{\alpha_{r}}^{a_{r}})$ is of boost-order $\leq -s-1.$ Therefore it follows that $(\mathcal{L}_{k}T)_{\alpha_{1}\cdots \alpha_{r}}=0,$ finishing this step. 

Now suppose it is true for integers up to $j-1$ with $j\geq 2.$ Since taking the Lie derivative by $k$ preserves boost-order, we know from the induction hypothesis that $(\mathcal{L}_{k})^{j}T$ has boost order $\leq s-j+1.$ If $\beta_{1}\cdots\beta_{r}$ is an index of b.w. $s-j+1$ then $((\mathcal{L}_{k})^{j}T)_{\beta_{1}\cdots\beta_{r}}$ can be written as a sum of $(k)^{j}(T_{\beta_{1}\cdots \beta_{r}})$ and terms of the form 
\begin{equation}
    ((\mathcal{L}_{k})^{q_1}T)_{a_1\cdots a_{r}}((\mathcal{L}_{k})^{q_1}(e_{\beta_1}\cdots e_{\beta_r}))^{a_{1}\cdots a_{r}},
\end{equation}
such that $q_1+q_2=j,$ with $q_1\leq j-1.$
The former is zero by assumption and the latter is zero by the fact that $(\mathcal{L}_{k})^{q_2}(e_{\beta_1}\cdots e_{\beta_r})$ has boost order $\leq j-s-1-q_2$ and the induction hypothesis which implies that $(\mathcal{L}_{k})^{q_1}T$ is of boost-order $\leq j-q_1$ . This proves that $(\mathcal{L}_{k})^{j}T$ is of boost-order $\leq s-j$, and therefore the claim is proven.
\newline

\begin{proposition}\label{CovProp}
Suppose that $(M,g,\lambda)$ is degenerate Kundt, $k$ is a Kundt vector field, $s$ is an integer and $T$ is a rank $r$ tensor. If $(\mathcal{L}_{k})^{j}T$ is of boost-order $\leq s-j$ w.r.t, for all $j\geq 0,$ then the same property is true for $\nabla T.$
\end{proposition}
\begin{proof}
Let $\{k,l,m_{i}\}$ be a degenerate Kundt frame for $(M,g,\lambda)$ completing $k$ on some open set. Working in this frame suppose that $\alpha_{0}\beta_{1}\cdots \beta_{r}$ is an index with b.w. $s-j+1$, for a given $j\geq 0.$ We have the identity
\begin{equation}\nabla_{\alpha_{0}}T_{\beta_{1}\cdots \beta_{r}}=e_{\alpha_{0}}(T_{\beta_{1}\cdots \beta_{r}})-\sum_{i=1}^r\tensor{\Gamma}{^\mu_{\alpha_{0}}_{\beta_{i}}}T_{\beta_{1}\cdots \mu\cdots \beta_{r}}.
\end{equation} 
Let us show that $(k)^j(e_{\alpha_{0}}(T_{\beta_{1}\cdots \beta_{r}}))$ and $(k)^j(\tensor{\Gamma}{^\mu_{\alpha_{0}}_{\beta_{i}}}T_{\beta_{1}\cdots \mu\cdots \beta_{r}})$ are both zero, then together with lemma \ref{componentlemma} this will prove the proposition.

Suppose first that $\alpha_{0}$ has b.w. $1$ then $\beta_{1}\cdots \beta_{r}$ is an index of b.w. $-j+s.$ In this case we have
\begin{equation}
   (k)^j(e_{\alpha_{0}}(T_{\beta_{1}\cdots \beta_{r}}))=(k)^{j+1}(T_{\beta_{1}\cdots \beta_{r}})=0,
\end{equation}
which vanishes by lemma \ref{componentlemma}.

Now suppose that $\alpha_{0}$ has b.w. $0,$ then $e_{\alpha_{0}}=m_{i}$ for some $i,$ and $\beta_{1}\cdots \beta_{r}$ has b.w. $s-j+1.$ By the construction of the frame we know that there exists a smooth function $f_{i}$ satisfying $k(f_i)=0$ such that $[k,m_{i}]=f_{i}k.$ It follows that
\begin{equation}
\begin{gathered}
    (k)^j(e_{\alpha_{0}}(T_{\beta_{1}\cdots \beta_{r}}))=(k)^jm_i(T_{\beta_{1}\cdots \beta_{r}})\\
    =m_i(k)^j(T_{\beta_{1}\cdots \beta_{r}})) +jf(k)^j(T_{\beta_{1}\cdots \beta_{r}})=0,
\end{gathered}
\end{equation}
where the last equality holds since $T_{\beta_{1}\cdots \beta_{r}}$ is a b.w. $s-j+1$ component of $T.$

Lastly, suppose that $\alpha_{0}$ has b.w. $-1,$ then $\beta_{1}\cdots\beta_{r}$ has b.w. $s-j+2,$ and $e_{\alpha_{0}}=l.$ By construction of the frame there exists $a,b^{i}\in C^{\infty}(U)$, $i=1,\dots n-2,$ satisfying $kk(a)=0$ and $k(b^{i})=0$ such that $[k,l]=ak+b^{i}m_{i}.$ Using the commutator relation we can write
\begin{equation}
\begin{gathered}
    (k)^j(e_{\alpha_{0}}(T_{\beta_{1}\cdots \beta_{r}}))=(k)^{j}l(T_{\beta_{1}\cdots \beta_{r}})\\
=(k)^{j-1}lk(T_{\beta_{1}\cdots \beta_{r}})+(j-1)k(a)(k)^{j-1}T_{\beta_{1}\cdots \beta_{r}} +a(k)^{j}(T_{\beta_{1}\cdots \beta_{r}}) \\
+b^{i}m_{i}(k)^{j-1}(T_{\beta_{1}\cdots \beta_{r}})+(j-1)b^{i}f_{i}(k)^{j-1}(T_{\beta_{1}\cdots \beta_{r}}).
\end{gathered}
\end{equation}
Continuing in a similar manner with repeated use of the commutator relation allows us to rewrite this expression into a sum of terms, each of which having at least $j-1$ derivatives of $T_{\beta_{1}\cdots \beta_{r}}$ with respect to $k$. It follows from lemma \ref{componentlemma} that all such terms vanish. 

Hence we have shown that $(k)^{j}(e_{\alpha_{0}}(T_{\beta_{1}\cdots \beta_{r}}))=0,$ whenever $\alpha_{0}\beta_{1}\cdots \beta_{r}$ is an index with b.w. $s-j+1$.
\newline

Now let us show that when $\alpha_{0}\beta_{1}\cdots \beta_{r}$ is an index with b.w. $s-j+1$ the terms \begin{equation}\label{exp}(k)^j(\tensor{\Gamma}{^\mu_{\alpha_{0}}_{\beta_{i}}}T_{\beta_{1}\cdots \mu\cdots \beta_{r}})\end{equation} vanish. Given a tensor $A$, we let $[A_{\alpha_{1}\cdots \alpha_{k}}]$ denote the boost-weight of a given component. Using this we have
\begin{equation}
    [T_{\beta_{1}\cdots \mu\cdots \beta_{r}}]=-j+s+1-[\tensor{\Gamma}{^{\mu}_{\alpha_{0}}_{\beta_{i}}}].
\end{equation}
$\tensor{\Gamma}{^{\mu}_{\alpha_{}}_{\beta_{0}}}$ is zero, for all strictly positive b.w. indices, therefore we need only show that the expression \eqref{exp} vanishes when $[\tensor{\Gamma}{^{\mu}_{\alpha_{0}}_{\beta_{i}}}]\leq 0.$ 

Now suppose that $[\tensor{\Gamma}{^{\mu}_{\alpha_{0}}_{\beta_{i}}}]=-t,$ for some integer $t\geq 0,$ then it follows from the construction of the frame that
\begin{equation}\label{SpecialChristoffel}
    (k)^{t+1}(\tensor{\Gamma}{^{\mu}_{\alpha_{0}}_{\beta_{i}}})=0.
\end{equation}
The expression \eqref{exp} can be written as a sum of terms of the form
\begin{equation}
    (k)^{n_1}(\tensor{\Gamma}{^\mu_{\alpha_{0}}_{\beta_{0}}})(k)^{n_2}(T_{\beta_{1}\cdots \mu\cdots \beta_{r}})
\end{equation}
such that $n_1+n_2=j.$ If $n_1\geq t+1$ then the term vanishes by \eqref{SpecialChristoffel}. Supposing that $n_1\leq t$, then $n_{2}\geq j-t=-[T_{\beta_{1}\cdots \mu\cdots \beta_{r}}]+s+1$, therefore such a term must vanish by lemma \ref{componentlemma}. This proves that $(k)^j(\tensor{\Gamma}{^\mu_{\alpha_{0}}_{\beta_{i}}}T_{\beta_{1}\cdots \mu\cdots \beta_{r}})=0,$ finishing the proof of the proposition.
\end{proof}

\begin{theorem}[Stability of algebraic type]\label{AlgebraicStability}
Suppose that $(M,g,\lambda)$ is degenerate Kundt and let $k\in \lambda$ be a Kundt vector field on some open set $U\subset M$. If $T$ is a rank $r$ tensor on $U$ and $s$ is some integer, then the following are equivalent:
\begin{enumerate}[i)]
    \item $(\mathcal{L}_{k})^jT$ is a tensor of boost-order $\leq s-j$, for all $j\geq 0.$
    \item $\nabla^{m}T$ has boost order $\leq s$, for all $m\geq 0.$
\end{enumerate}
\end{theorem}
\begin{proof}

$"i)\Rightarrow ii)"$ If $T$ satisfies the hypothesis of $i)$, then by proposition \ref{CovProp}, $\nabla T$ satisfies the same hypothesis. In particular $\nabla T$ is of boost-order $\leq s.$ We can continue this argument inductively, showing that $\nabla^{m}T$ if of boost-order $\leq s,$ for all $m\geq 0,$ proving the forward implication.

$"ii)\Rightarrow i)"$ Let $\{k,l,m_i\}=\{e_1,\dots,e_n\}$ be a degenerate Kundt frame completing $k$ on some open subset $V\subset U$, with  dual $\{w^{1},\dots w^{n}\}$. For such frames we know that if $w^{i_1}\cdots w^{i_r}$ is a tensor product of coframe elements having boost-weight $q$ then \begin{equation}(\mathcal{L}_{k})^{j} (w^{i_1}\cdots w^{i_r})\end{equation} has boost order $q-j,$ for each $j\geq 0.$ It therefore follows by the above forward implication, $"i)\Rightarrow ii)"$, that $\nabla^{m}(w^{i_1}\cdots w^{i_r})$ is of boost-order $\leq q,$ for all $m\geq 0.$ Writing
\begin{equation}
    T=T_{\alpha_{1}\cdots \alpha_{r}}w^{\alpha_{1}}\cdots w^{\alpha_{r}},
\end{equation}
we show inductively that if $m_1\geq 0$ and $\alpha_{1}\cdots \alpha_{r}$ has boost weight $s-m_1,$ then \begin{equation}\nabla^{m_1 +m_2+1}T_{\alpha_{1}\cdots \alpha_{r}}\end{equation} has boost-order $\leq m_1,$ for all $m_2\geq 0.$ 

Since $T$ and $\nabla T$ have boost-order $\leq s$ and
\begin{equation}
    \nabla T = (\nabla T_{\alpha_{1}\cdots \alpha_{r}})w^{\alpha_{1}}\cdots  w^{\alpha_{r}}+T_{\alpha_{1}\cdots \alpha_{r}}\nabla (w^{\alpha_{1}}\cdots w^{\alpha_{r}}).
\end{equation}
 it follows that if $\alpha_{1}\cdots\alpha_{r}$ has boost-weight $s,$ then $(\nabla T_{\alpha_{1}\cdots \alpha_{r}})$ is of boost-order $\leq 0$. This proves the statement for $(m_{1},m_{2})=(0,0).$
 
 Now suppose that $N\geq0$ is an integer such that claim is true for $(m_1,m_2)$ such that $m_{1}+m_2\leq N.$ The $(N+2)$-th covariant derivative $\nabla^{(N+2)}T$ can be expressed as a sum of terms 
 \begin{equation}\label{a}
 (\nabla^{N+2}T_{\alpha_{1}\cdots \alpha_{r}})w^{\alpha_{1}}\cdots w^{\alpha_{r}}
 \end{equation}
and 
\begin{equation}\label{aa}
     (\nabla^{q_{1}}T_{\alpha_{1}\cdots \alpha_{r}})\nabla^{q_{2}}(w^{\alpha_{1}}\cdots w^{\alpha_{r}})
\end{equation}
where $q_{1}\leq N+1$ and $q_{1}+q_{2}=N+2.$  

Suppose that $\alpha_{1}\cdots \alpha_{r}$ is an index of b.w. $s-t$, for some $t\geq 0.$ If $q_{1}\leq t,$ then the expression \eqref{aa} is of boost-order $\leq s$ since each successive covariant derivative raises the boost-order by a value of at most one. If $q_{1}>t$ then $q_{1}=t+(q_{1}-t-1)+1$ with $t+(q_{1}-t-1)\leq N.$ Therefore the induction hypothesis implies that $(\nabla^{q_{1}}T_{\alpha_{1}\cdots \alpha_{r}})$ has boost-order $\leq t.$ This shows that for any indices $\alpha_{1}\cdots \alpha_{r}$ the expression \eqref{aa} has boost-order $\leq s.$ 

Since by assumption $\nabla^{N+2}T$ has boost-order $\leq s,$ it follows that the terms of the form \eqref{a} must also be of boost-order $\leq s.$ Hence if $\alpha_{1}\cdots \alpha_{r}$ is an index of b.w. $s-t$, for some $t\geq 0,$ then
\begin{equation}
    (\nabla^{N+2}T_{\alpha_{1}\cdots \alpha_{r}})=(\nabla^{t+(N+1-t)+1}T_{\alpha_{1}\cdots \alpha_{r}})
\end{equation}
has boost-order $\leq t.$ Thus the statement is true for $(m_{1},m_{2})$ such that $m_{1}+m_{2}\leq N+1.$ By induction this shows that the statement holds true for any pair of integers $(m_{1},m_{2}).$

In particular, if $\alpha_{1}\cdots\alpha_{r}$ has b.w. $s-t+1$ for some $t\geq 0,$ then $\nabla^{t}T_{\alpha_{1}\cdots \alpha_{r}}$ is of boost-order $\leq t-1,$ implying that
\begin{equation}
    (\nabla^{t}T_{\alpha_{1}\cdots \alpha_{r}})(\underbrace{k,\dots,k}_{t})=(k)^t(T_{\alpha_{1}\cdots \alpha_{r}})=0.
\end{equation}
By lemma \ref{componentlemma} the tensor $T$ must therefore satisfy the hypothesis of $i)$ on $V.$ This finishes the proof.
%



    

\end{proof}
\section{SPI-preserving deformations}\label{spipreservingdefs}
In \cite{CharBySpi} and \cite{SCPI} the authors study deformations of Lorentzian metrics in which the spi's of the curvature tensor and all its covariant derivatives stay unchanged throughout the deformation. They give the following definition:

\begin{definition}\label{Ideg}
	A Lorentzian metric $g$ on a manifold $M$ is said to be $\mathcal{I}$-degenerate if there exists a smooth family of metrics $g_{t},$ for t in some interval $[0,\epsilon)$, such that the following hold:
	\begin{enumerate}[i)]
		\item $g_{0}=g.$
		\item $g_{t}$ is not isometric to $g_{0},$ for all $t>0.$
		\item The spi's of $g_t$ are identical, for each $t.$
	\end{enumerate}
\end{definition}
In particular the authors have shown that any four dimensional Lorentzian metric $g$ which is $\mathcal{I}$-degenerate must be of the degenerate Kundt class, which has been discussed in the previous section. In this case the deformation $g_t$ achieving the $\mathcal{I}$-degeneracy goes in the direction of type $III$ tensors with respect to $g$ and $\lambda,$ i.e., there exists a family of tensors $T_t$ of type $III$ such that $g_t= g+T_t$, for all $t.$


In this section we shall seek to understand the deformations in the direction of type $III$ of metrics whose curvature tensors are of type $II.$ We will focus on deformations which stay in the degenerate Kundt class in order that the curvature and its covariant derivative to all orders stay of type $II$ throughtout the deformation. In addition to being interesting in its own right, we shall also use the results in order to understand which nil-Killing vector fields preserve the scalar polynomial curvature invariants of $g.$

Let $(M,g,\lambda)$ be a Lorentzian manifolds along with a null-distribution. Motivated by the discussion above we now consider a deformation $g_t$ such that the deformation tensor $h_s:=\frac{\partial}{\partial t}g_t$ if of type $III$ with respect to $(g,\lambda),$ for each $t.$ Such deformations have a few amenable properties that we point out:

The difference $(g_t-g)$ is of type $III$ with respect to $(g,\lambda).$ Hence $\lambda$ remains a null-distribution for each of the metrics $g_t.$ Furthermore the distribution $\lambda^{\perp}$ remains invariants with respect to which metric the orthogonal compliment is defined. In particular the algebraic classification of tensors is identical for each triple $(M,g_t,\lambda).$ Moreover, for a tensor $P$ of type $II$ with respect to $\lambda,$ the spi's of $P$ are the same when taken with respect to any of the metrics $g_t	$.

Below we list  evolution equations \cite{Hamiltons,Topping} for the Ricci scalar  and the Riemann curvature tensor during the deformation. The dependence on the parameter $t$ is implicit.



\begin{equation}\label{RSEvolution}
    \partial _{t} R = \tensor{h}{_a_b}\tensor{R}{^a^b}-\nabla^{a}\nabla^{b}\tensor{h}{_a_b} + \Delta \tensor{h}{^{a}_{a}},
\end{equation}

\begin{equation}\label{RiemannDeformation}
\begin{gathered}
(\frac{\partial}{\partial t}R)_{abcd}=\frac{1}{2}[\tensor{R}{^f_b_c_d}h_{fa}-\tensor{R}{^f_a_c_d}h_{fb}]\\
+\frac{1}{2}[\nabla_{c}\nabla_bh_{ad}-\nabla_{d}\nabla_{b}h_{ac} + \nabla_{d}\nabla_{a}h_{bc}-\nabla_{c}\nabla_{a}h_{bd}].
\end{gathered}
\end{equation}


Furthermore for $m\geq 1$
\begin{equation}\label{InductiveDeformation}
\begin{gathered}
\partial_{s}(\nabla^{m}Rm)=(\partial_{s}\nabla)(\nabla^{m-1}Rm)+\nabla\partial_{s}(\nabla^{m-1}Rm)\\
=\nabla h * \nabla^{m-1}Rm + \nabla(\partial_{s}\nabla^{m-1}Rm).
\end{gathered}
\end{equation}
The notation $T*S$ between two tensors $T$ and $S$ is shorthand for some linear combination of traces of $T\otimes S$ giving the correct rank.

\begin{theorem}\label{DefTheorem}
Suppose that $(M,g,\lambda)$ is degenerate Kundt. Let $g_{t}$ be a deformation with $h_t:=\frac{\partial}{\partial t}g_t$ such that w.r.t. $\lambda$, $h_{t}$ is of type $III$ and  $\mathcal{L}_{k}h_{t}$ has boost-order $\leq -2$ for each Kundt vector field $k$, for all $t.$
Then the following properties are equivalent:
\begin{enumerate}[i)]
\item $(\mathcal{L}_{k})^2h_{t}=0$ for all $t$ and each Kundt vector field $k$.
\item $\frac{\partial}{\partial t}(\nabla^{m}R_{t})$ is of type $III$ w.r.t to $\lambda$, for all $t$ and $m\geq 0$.
\item All spi's are preserved throughout the deformation.
\end{enumerate}
\end{theorem}
\begin{proof}
$"i)\Rightarrow ii)"$ First we show that $g_t$ is degenerate Kundt with respect to w.r.t. $\lambda,$ for all $t.$ Let $k$ be a Kundt vector field of $(M,g,\lambda)$ belonging to $\lambda$. 

By assumption we see that 
\begin{equation}\mathcal{L}_{k}h_{t}=\frac{\partial}{\partial t}\mathcal{L}_{k}g_t\end{equation} is of boost-order $\leq -2,$ for all $t.$ Therefore $\mathcal{L}_{k}g_{t}-\mathcal{L}_{k}g$ is of boost-order $\leq -2,$ and hence $\mathcal{L}_{k}g_{t}$ is of type $III$ w.r.t. $\lambda,$ for all $t.$ Hence $k$ is also a Kundt vector field for $(M,g_t,\lambda),$ showing that the triple is Kundt. Furthermore $i)$ implies that $\frac{\partial}{\partial t}(\mathcal{L}_{k})^2g_t=(\mathcal{L}_{k})^2h_t=0,$ showing that \begin{equation}
    (\mathcal{L}_{k})^2g_t=(\mathcal{L}_{k})^2g,
\end{equation}
for all $t.$ By proposition \ref{DegkundtChar} it follows that $(M,g_t,\lambda)$ is degenerate Kundt w.r.t. $\lambda, $ for all $t.$

Since each $g_t$ is degenerate Kundt w.r.t. $\lambda$  and $(\mathcal{L}_{k})^2h_t=0$ for any Kundt vector field $k,$ theorem \ref{AlgebraicStability} implies that $\nabla^{m}h_{t}$ is of type $III$ w.r.t. $\lambda$ for all $m\geq 0$. Furthermore $\nabla^{m} R_{t}$ is of type $II$ w.r.t. $\lambda$ for all $m\geq 0$ and $t.$ It follows from equations \eqref{RiemannDeformation} and \eqref{InductiveDeformation} that $\frac{\partial}{\partial t}(\nabla^{m}R)$ is of type $III,$ for all $m\geq 0.$ 

$"ii)\Rightarrow iii)"$ By the assumptions of $ii)$ it follows that
\begin{equation}
    \nabla^{m}R_t-\nabla^{m}R_0,
\end{equation}
is of type $III$ and therefore  $\nabla^{m}R_t$ is of type $II$ w.r.t. $\lambda,$ for all $m\geq 0$ and $t$.
Since $h_{t}$ is of type $III,$ raising the index of a curvature tensor $\nabla^{m}R_t$ by using $g_t$ gives a tensor whose variation $\frac{\partial}{\partial t}(g_t*\nabla^{m}R_t)$ is of type $III$ w.r.t. $\lambda.$ In particular, it follows from the discussion in section $1$ that the values of each spi of $g_t$ remain fixed throughout the deformation.

$iii)\Rightarrow i)$ By proposition \ref{DegkundtChar} our assumptions show that the curvature $Rm_t$ of $g_t$ is of type $II$ w.r.t. $\lambda,$ for all $t.$ Since $h_t$ is of type $III$ w.r.t. $\lambda$ it follows from equation \eqref{RSEvolution} and the constancy of the Ricci scalar throughout the deformation that \begin{equation}(\nabla^{t})^{a}(\nabla^{t})^{b}(h_t)_{ab}=0.\end{equation} Since $(M,g_t,\lambda)$ is Kundt for all t, we can find a family of frames $\{k_t,l_t,(m_{t})_{i}\}$  such that $k_t$ is a Kundt vector field for $(M,g_t,\lambda)$ and
\begin{equation}
    \nabla^{t}_{k_{t}}(m_t)_{i}=0,\quad \nabla^{t}_{k_t}l_t=0,
\end{equation}
for all $t$ and $i.$ We can find families of functions $a_t$ and $b^{i}_t$ such that the deformation tensor is given by
\begin{equation}
    h_t = a_t k_t\otimes_{S} k_t + b^{i}_{t}(k_t\otimes_{S} (m_{t})_{i}), 
\end{equation}
for all $t.$ Since $\mathcal{L}_{k_t}h_t$ is of boost-order $\leq -2,$ we see that $k_t(b_t^{i})=0,$ for all $i$ and $t.$ By the Kundt property it therefore follows that 
\begin{equation}
    0=(\nabla^{t})^{a}(\nabla^{t})^{b}(h_t)_{ab}=\nabla^{t}_{k_t}\nabla^{t}_{k_t}a_t,
\end{equation}
which shows that $(\mathcal{L}_{k})^2h_t=0,$ for every Kundt vector field $k.$
\end{proof}

Recall \cite{KundtSpacetimes} that if $(M,g,\lambda)$ is degenerate Kundt we can locally express the metric by
\begin{equation}\label{DegkundtCoordinates}
    g=2du(dv+Hdu + W_{i}dx^{i}) + \tilde{g}_{ij}dx^{i}dx^{j}
\end{equation}
where $\frac{\partial}{\partial v}$ is a Kundt vector field belonging to $\lambda$ and the functions $H,W_{i}$ take the form
\begin{equation}
    H(u,v,x^{k})=v^2H^{(2)}(u,x^k)+vH^{(1)}(u,x^k)+H^{(0)}(u,x^k)
\end{equation}
and
\begin{equation}
W_{i}(u,v,x^{k})=vW^{(1)}_{i}(u,x^{k}) +W^{(0)}_{i}(u,x^{k}).
\end{equation}
Now suppose that $g_t$ is a deformation of $g$ with deformation tensor $h_t=\frac{\partial}{\partial t}g_t$ satisfying $i),ii)$ and $iii)$ of theorem \ref{DefTheorem}. Then there exists families of smooth function $P^{(j)}(t,u,x^k),$ for $j=0,1,$ and $Q_{i}(t,u,x^{k})$ for $i=1,\dots n-2$ such that in coordinates
\begin{equation}\label{DeformationCoordinates}
    h_t = du[(vP^{(1)}(t,u,x^{k})+P^{(0)}(t,u,x^k))du + Q_{i}(t,u,x^{k})dx^{i}],
\end{equation}
for all $t.$ These are the local deformations that preserve the values of the spi's which have been considered in \cite{SCPI}. Our work has shown that deformations of this type are exactly those such that the deformation tensor and all its covariant derivatives are of type $III$ with respect to $\lambda.$

\begin{remark}[Existence of deformations preserving spi's for degenerate Kundt spacetimes]
The expression \eqref{DeformationCoordinates} shows if $(M,g,\lambda)$ is degenerate Kundt, then we can always deform the metric on a neighborhood of any given point in a manner satisfying the hypothesis of theorem \ref{DefTheorem}, ensuring that the spi's are preserved. In order to use this for ensuring the existence of a deformation $g_t$ of the same nature on the whole manifold $M,$ we can attempt to proceed as follows:

Suppose we have a locally finite $\{U_{\alpha}\}_{\alpha\in I}$ covering of $M$ such that each $U_{\alpha}$ admits coordinates of the form \eqref{DegkundtCoordinates}. On each such neighborhood we choose a deformation $g^{\alpha }_{t}$ of $g$ restricted to $U_{\alpha}$ satisfying the hypothesis of theorem \ref{DefTheorem}. Now take a partition of unity, $\{\psi^{\alpha}\}_{\alpha\in I}$, subordinate to the covering $\{U_{\alpha}\}_{\alpha\in I}$. We construct a global deformation $g_t$ on $M$ by letting
\begin{equation}
    g_t:=\sum_{\alpha\in I}\psi^{\alpha}g^{\alpha}_{t}.
\end{equation}
We see that $g_t$ gives a well-defined Lorentzian metric for each $t$ since $g_t^\alpha-g$ is a tensor of type $III$, for all $t$ and $\alpha.$ Moreover the deformation itself will be such that the deformation tensor 
\begin{equation}
    h_{t}=\sum_{\alpha\in I}\psi^{\alpha}h^{\alpha}_{t},
\end{equation}
is a tensor of type $III$ w.r.t. $\lambda$, for all $t,$ and thus $(M,g_t,\lambda)$ is Kundt and the whole family have the same collection of Kundt vector fields. However we see that if $k$ is a Kundt vector field of $g$ belonging to $\lambda$, then
\begin{equation}
    \mathcal{L}_{k}h_{t}=\mathcal{L}_{k}(\sum_{\alpha\in I}\psi^{\alpha}h^{\alpha}_{t})=\sum_{\alpha\in I}[(\mathcal{L}_{k}\psi^\alpha)h_t^\alpha+\psi^\alpha\mathcal{L}_{k}h_t^\alpha].
\end{equation}
In general $k(\psi^{\alpha})$ must vanish in order for this to be a tensor of boost-order $\leq -2.$ Thus in order to prove global existence by this method, we must ensure the existence of a partition of unity subordinate to the covering such that $\psi^\alpha$ is constant along the integral curves of $\lambda,$ for all $\lambda.$ For instance this shows that if $\alpha\in I$, then $U_{\alpha}$ by necessity must contain each maximal integral curve of $\lambda$ it intersects. Therefore, given a degenerate Kundt spacetime we conclude the following: Even though we can always produce local deformations of the metric which preserve the values of the spi's, such deformations might not exist globally.
\end{remark}


Let $(M,g,\lambda)$ be degenerate Kundt. Now we investigate deformations $g_t$ which can be achieved as a pull-back of the metric by the one-parameter group of diffeomorphisms of a vector field such that the deformation tensor $h_t$ is of type $III$ w.r.t $(M,g,\lambda)$, for all $t.$ Given a vector field $X$ with flow $\phi_t$ the deformation tensor $h_t$ of $g_t:=(\phi_{t})^*g$ takes the form
\begin{equation}
    h_t=\frac{\partial}{\partial t}((\phi_{t})^*g)=\lim_{s\rightarrow 0}\frac{1}{s}((\phi_{t+s})^*g-(\phi_t)^*g)=\phi_{t}^{*}\mathcal{L}_{X}g.
\end{equation}
Hence $h_t$ is of type $III$ with respect to $(M,g,\lambda)$, for all $t,$ iff. $X$ is nil-Killing with respect to $\lambda$ and $[X,\lambda]\in\lambda.$ Such vector fields $X$ are exactly the ones belonging to the Lie algebra $\mathfrak{g}_{\lambda}$ from section \ref{Lie algebra}. We can therefore use the local expression of nil-Killing vector fields preserving algebraic structure found in proposition \ref{NilKillinInCoordinates} in order to show the following: Among the local deformations \eqref{DeformationCoordinates} on the open set $U$ we can always find ones which give $\mathcal{I}$-degeneracies, in the sense that they immediately leave the orbit of $g$ on $U.$ 

Before we set out to do this, we present some results giving conditions for a vector field to preserve the spi's. For now we can drop the assumption that $X$ must preserve algebraic structure.

Suppose again that $(M,g,\lambda)$ is degenerate Kundt. Let $X$ be a vector field with flow $\phi_t$ such that $X$ is nil-Killing w.r.t. $\lambda$ and $g_t=(\phi_t)^*g$ with deformation tensor $h_t=\frac{\partial}{\partial t}g_t$. Then by definition
\begin{equation}
    h_{0}=\mathcal{L}_{X}g,
\end{equation}
which is of type $III$ w.r.t $\lambda$. Furthermore by the invariance of the covariant derivative and the Riemann tensor we see that
\begin{equation}
    \mathcal{L}_{X}(\nabla^{m}Rm)=\frac{\partial}{\partial t}\vert_{t=0}(\nabla_{t})^mRm_t.
\end{equation}
In particular we can substitute $\mathcal{L}_{X}g$ for $h_{ab}$ in the expressions \eqref{RiemannDeformation} and \eqref{InductiveDeformation} in order to find the expression for $\mathcal{L}_{X}(\nabla^{m}Rm)$.

\begin{proposition}\label{EasySpiPreservation}
Suppose that $(M,g,\lambda)$ is degenerate Kundt. Let $X$ be a vector field on $M$ satisfying the following properties for any kundt vector field $k$:
\begin{enumerate}[i)]
    \item $X$ is nil-Killing w.r.t. $\lambda.$
    \item $\mathcal{L}_{k}\mathcal{L}_Xg$ is of boost-order $\leq -2$ w.r.t. $\lambda.$
\end{enumerate}
Then $X(\mathcal{I})=0$, for all scalar curvature invariants $\mathcal{I}$ of $g,$ iff. $(\mathcal{L}_{k})^2\mathcal{L}_{X}g=0,$ for every Kundt vector field $k.$
\end{proposition}
\begin{proof}
By the above comments this proposition can be proved by substituting $\mathcal{L}_{X}g$ for $h_{0}$ and following the proof of theorem \ref{DefTheorem}.
\end{proof}

Suppose that $(M,g,\lambda)$ is a Kundt spacetime and let $\mathfrak{g}_{\lambda}$ be the Lie algebra of algebra preserving nil-Killing vector fields w.r.t $\lambda$. We define
\begin{equation}\label{spipreservingLiealg}
    \mathfrak{h}_{\lambda}
\end{equation}
to be the
collection of vector fields $X\in \mathcal{X}(M)$ which satisfy the following properties:
\begin{enumerate}[i)]
    \item $X\in \mathfrak{g}_{\lambda}.$
    \item $\mathcal{L}_{k}\mathcal{L}_{X}g$ is of boost-order $\leq -2$ w.r.t. $\lambda,$ for each Kundt vector field $k.$
    \item $(\mathcal{L}_{k})^2\mathcal{L}_{X}g=0,$ for each Kundt vector field $k.$
\end{enumerate}
\begin{proposition}
Suppose $(M,g,\lambda)$ is a degenerate Kundt spacetime. Then $\mathfrak{h}_{\lambda}$ is a Lie subalgebra of $\mathfrak{g}_{\lambda}$.
\end{proposition}
\begin{proof}
Suppose $X,Y\in \mathfrak{h}_{\lambda}$. If $k$ is a Kundt vector field, then
\begin{equation}
    \mathcal{L}_{k}\mathcal{L}_{X}\mathcal{L}_{Y}g=\mathcal{L}_{[k,X]}\mathcal{L}_{Y}g-\mathcal{L}_{X}(\mathcal{L}_{k}\mathcal{L}_{Y}g).
\end{equation}
Since $[k,X]$ is a Kundt vector field and the operator $\mathcal{L}_{X}$ preserves algebraic type, it follows that $\mathcal{L}_{k}\mathcal{L}_{X}\mathcal{L}_{Y}g$ is of boost-order $\leq -2$ w.r.t. $\lambda.$ Furthermore
\begin{equation}
    \begin{gathered}
    (\mathcal{L}_{k})^2\mathcal{L}_{X}\mathcal{L}_{Y}g=\mathcal{L}_{k}\mathcal{L}_{[k,X]}\mathcal{L}_{Y}-\mathcal{L}_{k}\mathcal{L}_{X}\mathcal{L}_k\mathcal{L}_{Y}g\\
    =\mathcal{L}_{k}\mathcal{L}_{[k,X]}\mathcal{L}_{Y}g-\mathcal{L}_{[k,X]}\mathcal{L}_k\mathcal{L}_{Y}g+\mathcal{L}_{X}(\mathcal{L}_{k})^2\mathcal{L}_{Y}g=0.
    \end{gathered}
\end{equation}

The same is true when we switch the order of $X$ and $Y$ and therefore the Jacobi identity shows that
\begin{equation}
    \mathcal{L}_{k}\mathcal{L}_{[X,Y]}g=\mathcal{L}_{k}\mathcal{L}_{X}\mathcal{L}_{Y}g-\mathcal{L}_{k}\mathcal{L}_{Y}\mathcal{L}_{X}g
\end{equation}
is of boost-order $\leq -2$, w.r.t. $\lambda$, and
\begin{equation}
    (\mathcal{L}_{k})^2\mathcal{L}_{[X,Y]}g=(\mathcal{L}_{k})^2\mathcal{L}_{X}\mathcal{L}_{Y}g-(\mathcal{L}_{k})^2\mathcal{L}_{Y}\mathcal{L}_{X}g=0.
\end{equation}
Hence $[X,Y]\in \mathfrak{h}_{\lambda}.$
\end{proof}

Let us give a coordinate presentation of nil-Killing vector fields satisfying the assumptions of proposition \ref{EasySpiPreservation} with the added assumption that $X$ preserves the algebraic structure. Suppose that $(M,g,\lambda)$ is degenerate Kundt and let $U$ be an open set such that there exists coordinates $(u,v,x^{k})$ such that the metric $g$ takes the form \eqref{DegkundtCoordinates} on $U.$ By proposition \ref{CoordinateNilkilling} a vector field $X$ is nil-Killing w.r.t. $\lambda$ and satisfies $[X,\lambda]\subset \lambda$ iff. there exists function $A(u),$ $B(u,x^k)$ and $C^{i}(u,x^{k})$, for $i=1,\dots n-2$ such that
\begin{equation}
    X=A(u)\frac{\partial}{\partial u}+(-v\frac{\partial A}{\partial u}(u)+B(u,x^{k}))\frac{\partial}{\partial v}+C^{i}(u,x^{k})\frac{\partial}{\partial x^{i}}
\end{equation}
where
\begin{equation}\label{TransverseKilling}
    C^{k}\frac{\partial \tilde{g}_{ij}}{\partial x^{k}} + \tilde{g}_{jk}\frac{\partial C^{k}}{\partial x^{i}}+\tilde{g}_{ik}\frac{\partial C^{k}}{\partial x^{j}} + A\frac{\partial \tilde{g}_{ij}}{\partial u}=0.
\end{equation}
Plugging this back into $\mathcal{L}_{X}g$ we get
\begin{equation}
    \mathcal{L}_Xg=\Lambda duu +\Omega_{i}dudx^{i}
\end{equation}
where
\begin{multline}\label{LieDeriv1}
    \Lambda=2\Big{\{}X(H^{(2)})v^2+[\frac{\partial}{\partial u}(AH^{(1)})+C^{i}\frac{\partial}{\partial x^{i}}H^{(1)}+2BH^{(2)}+W_{i}^{(1)}]v\\
    +A\frac{\partial}{\partial u}H^{(0)}+2H^{(0)}\frac{\partial}{\partial u}A+BH^{(1)}+C^{i}\frac{\partial}{\partial x^{i}}H^{(0)}+\frac{\partial}{\partial u}B + W_{i}^{(0)}\frac{\partial}{\partial u}C^{i}\Big{\}}
\end{multline}
and
\begin{multline}
    \Omega_{i}=2\Big{\{}{v[A\frac{\partial}{\partial u}W_{i}^{(1)}+C^{j}\frac{\partial}{\partial x^{j}}W_{i}^{(1)}+W_{j}^{(1)}\frac{\partial}{\partial x^{i}}C^{j}]}
    +A\frac{\partial}{\partial u}W_{i}^{(0)}+BW_{i}^{(1)}\\ + C^{j}\frac{\partial}{\partial x^{j}}W_{i}^{(0)} +\frac{\partial}{\partial x^{i}}B + W_{j}^{(0)}\frac{\partial}{\partial x^{i}}C^{j}+W_{i}^{(0)}\frac{\partial}{\partial u}A + \tilde{g}_{ij}\frac{\partial}{\partial u}C^{j}\Big{\}}.
\end{multline}
Hence in these coordinates we see that a vector field which preserves algebraic structure and is nil-Killing w.r.t. $\lambda$ satisfies the hypothesis of \ref{EasySpiPreservation}, i.e., $\mathcal{L}_{k}\mathcal{L}_{X}g$ is of boost-order $\leq -2$ and $(\mathcal{L}_{k})^{2}\mathcal{L}_{X}g=0,$ if and only if 
\begin{equation}\label{nilKillingH2}
    X(H^{(2)})=0\quad \text{ and } \quad A\frac{\partial}{\partial u}W_{i}^{(1)}+C^{j}\frac{\partial}{\partial x^{j}}W_{i}^{(1)}+W_{j}^{(1)}\frac{\partial}{\partial x^{i}}C^{j}=0.
\end{equation}

\begin{proposition}
Suppose $(M,g)$ is a Lorentzian manifold such that if $V^{\prime}\subset M$ is open we can find an open set $V\subset V^{\prime}$ for which the collection of null-distributions $\lambda$ on $V$ such that $(V,g,\lambda)$ is denerate Kundt, is non-empty and finite. Then each point of $M$ has a neighborhood $U$ with a deformation $g_t$ of $g$ restricted to $U,$ such that $g_t$ is an $\mathcal{I}$-degeneracy in the sense of definition \ref{Ideg} on U.
\end{proposition}
\begin{proof}
Given a point $p\in M,$ let $V$ be neighborhood such that there exists a null distribution $\lambda,$ for which $(U,g,\lambda)$ is degenerate Kundt. Now find a  neighborhood $\tilde{V}\subset V$ of $p$ with coordinates $(u,v,x^{k})$ such that $\frac{\partial}{\partial v}\in \lambda,$ and the local expression of the metric $g$ is given by \eqref{DegkundtCoordinates}. If necessary shrink to a neighborhood $U\subset \tilde{V}$ of $p$ for which there is a finite number of  degenerate Kundt null-distribution  and restrict the coordinates to $U$. Find functions $P(u,x^{k})$ and $Q(u,x^{k})$ such that if $X$ is a vector field on $U$ which preserves algebraic structure and is nil-Killing w.r.t. $\lambda,$
then 
\begin{equation}
    \mathcal{L}_{X}g\neq (vP(u,x^k)+Q(u,x^k))dudu.
\end{equation}
By inspection of expressions \eqref{TransverseKilling}, \eqref{LieDeriv1} and \eqref{nilKillingH2} such a pair always exists.

Now let $g_{t}$ be the deformation of $g$ given by
\begin{equation}\label{subequation}
    g_t=g+t(vP+Q)dudu.
\end{equation}
By theorem \ref{DefTheorem} the spi's of the metrics stay fixed along the deformation.

Now suppose that $g_t$ does not define an $\mathcal{I}$-degeneracy on $U$. Then there exists a family of diffeomorphisms $\{\psi_{t}\}_{t\in I}$ of $U$ defined for some interval $I$ about $0,$ such that $g_t=(\psi_{t})^{*}g$, for all $t\in I.$ Equation \eqref{subequation} implies that $\psi_{t}$ is an isometry on $\lambda^{\perp}.$ Letting $k=\frac{\partial}{\partial v}$, then if $t\in I,$ we see that
\begin{equation}
    (\psi_{t})^{*}(\mathcal{L}_{(\psi_{t}^{-1})^*k}g)=\mathcal{L}_{k}(\psi_{t}^{*}g)=\mathcal{L}_{k}g+tPdudu.
\end{equation}
It follows that $(U,g,(\psi_{t}^{-1})_{*}\lambda)$ is Kundt. Moveover
\begin{equation}
    (\psi_{t})^{*}(\mathcal{L}_{(\psi_{t}^{-1})^*k})^2g =(\mathcal{L}_{k})^2(\psi_{t}^{*}g)=(\mathcal{L}_{k})^2g
\end{equation}
and similarly
\begin{equation}
    (\psi_{t})^{*}(\mathcal{L}_{(\psi_{t}^{-1})^*k})^3g=(\mathcal{L}_{k})^3(\psi_{t}^{*}g)=(\mathcal{L}_{k})^{3}g=0,
\end{equation}
from which it follows that $(U,g,(\psi_{t}^{-1})_{*}\lambda)$ is degenerate Kundt. By finiteness of the collection of degenerate Kundt null-distributions we see that $(\psi_{t})_{*}(\lambda)=\lambda$ and therefore $\psi_{t}$ preserves algebraic structure, for all $t.$ Thus it is clear that $\{\psi_{t}\}_{t\in I}$ belongs to the group $\mathcal{G}_{g}^{-1}$ from section \ref{Lie algebra} and therefore we can find a vector field $X\in\mathfrak{g}^{-s}_{g}$, which by definition is nil-Killing w.r.t. $\lambda$ and preserves algebraic structure, such that $\mathcal{L}_{X}g=\frac{\partial}{\partial t}\psi_{t}^{*}g=(Pv+Q)dudu.$ This gives a contradiction, showing that $g_t$ is an $\mathcal{I}$-degeneracy.
\end{proof}

\section{Kundt-CSI spacetimes}
Recall that a Lorentzian manifold $(M,g)$ is said to be CSI (Constant scalar invariants), if all the scalar polynomial curvature invariants of $g$ are constant across the manifold. The analogous property for a Riemannian manifold would force it to be locally homogeneous by results in \cite{CurvatureInvariantsRiemann}. A Pseudo-Riemannian manifold is locally homogeneous iff. about each point we can find a transitive collection of local Killing vector fields.

We wish to generalize this property to CSI spacetimes by characterizing them as those Lorentzian manifolds which at each point have a transitive collection of vector fields generalizing properties of Killing vector fields. In short we wish to define a Lie algebra of vector fields generalizing the property of Killing vector fields in the sense that they preserve the spi's of a metric.

It was suggested in \cite{herviknew} that nil-Killing vector fields might provide a suitable class of vector fields for such a characterization. In order for the collection of them to constitute a Lie algebra we need the added assumption that they preserve algebraic structure. We have the following proposition which was first proved in \cite{IDIFF} and for which we give a simpler proof:
\begin{proposition}\label{CSILocallytransitive}
Suppose that $(M,g,\lambda)$ is a Kundt spacetime which is CSI. About each point $p\in M$ we can find a locally transitive collection, $\{X_{i}\}_{i\in I}$, of algebra preserving nil-Killing vector fields w.r.t. $\lambda$.
\end{proposition}
\begin{proof}
Suppose that $p\in M.$ Find local coordinates $(u,v,x^{k})$ about $p$ with $k:=\frac{\partial}{\partial v}\in \lambda$ and $k^{\natural}=du.$ Then writing
$g=2du(dv+Hdu+W_{i}dx^{i})+\tilde{g}_{ij}(u,x^{k})dx^{i}dx^{j}$ it follows from a result in \cite{CSI}, that $\tilde{g}_{ij}(u,x^{k})$ is independent of $u$ and locally homogeneous. Therefore by the results of proposition \ref{Kundttransverse}, we can find a locally transitive collection of algebra preserving nil-Killing vector fields w.r.t. $\lambda.$
\end{proof}
Note however that algebra preserving nil-Killing vector fields do not a priori preserve spi's and therefore the converse to this proposition is not true.

 The added assumptions of proposition \ref{EasySpiPreservation} ensure that such vector fields preserve the spi's of the metric. 
However the nil-Killing vector fields which preserve the algebraic structure given by this proposition cannot always give a transitive collection. This can be seen as follows: Consider a degenerate Kundt manifold $(M,g,\lambda)$ and find local coordinates $(u,v,x^k)$ such that the metric is given by the expression in \eqref{DegkundtCoordinates}. If we could find a locally transitive collection of nil-Killing vector fields satisfying the hypothesis of proposition \ref{EasySpiPreservation}, then by \eqref{nilKillingH2}, the function $H^{(2)}(u,x^{k})$ must be constant. However, there are many degenerate Kundt metrics which are CSI for which $H^{(2)}$ is not contant. A simple example can be found by considering the VSI (vanishing scalar invariants) manifold constructed in \cite{HigherDimVsi}, given by

\begin{equation}
    g=2du(dv +Hdu+W_{i}dx^{i})+\delta_{ij}dx^{i}dx^{j},
\end{equation}
where
\begin{equation}
    H(u,v,x^{k})=\frac{v^2}{2(x^1)^2} + vH^{(1)}(u,x^{k})+H^{(0)},
\end{equation}
\begin{equation}
    W_{1}(u,v,x^{k})=-\frac{2v}{x^1},\quad W_{i}(u,v,x^{k})=W_{i}^{0}(u,x^{k}),\quad i\neq 1.
\end{equation}
The metric $g$ is degenerate Kundt with spi's that vanish identically, and yet \begin{equation}H^{(2)}(u,x^{k})=\frac{1}{2(x^1)^2}.\end{equation}

However for degenerate Kundt spacetimes $(M,g,\lambda)$ for which $\lambda$ admits recurrent null vector fields such a characterization is possible:
\begin{proposition}
Suppose $(M,g,\lambda)$ is a degenerate Kundt spacetime such that about each point in $M$ there exists a local null vector field $k\in \lambda$ and a one-form $\omega,$ for which $\nabla k=\omega\otimes k.$ Then $(M,g,\lambda)$ is $CSI$ iff. about any $p\in M$ there  exists a locally transitive collection $\{X_{i}\}_{i\in I}$ of vector fields such that for $i\in I$, the following are satisfied:
\begin{enumerate}[i)]
    \item  $X_{i}$ is an algebra-preserving nil-Killing vector field w.r.t $\lambda,$ 
    \item $\mathcal{L}_{k}\mathcal{L}_{X_{i}}g$ is of boost-order $\leq -2$ w.r.t. $\lambda,$
    \item $(\mathcal{L}_{k})^2\mathcal{L}_{X_{i}}g=0,$
\end{enumerate}
for each Kundt vector field $k.$
\end{proposition}
\begin{proof}
Since $\lambda$ contains local vector fields which are recurrent, it follows that about each point $p\in M$ there exists a coordinates $(u,v,x^{k})$ such that $k:=\frac{\partial}{\partial v}\in \lambda$ and $k^{\natural}=du$ and $\mathcal{L}_{k}g$ is of boost-order $\leq -2$ w.r.t. $\lambda$.

Now suppose that $X$ is an algebra-preserving nil-Killing vector field w.r.t. $\lambda.$ Expressing $X$ by the coordinate  expression \eqref{NilKillinInCoordinates} it is clear that $[k,X]=fk$ for some smooth function $f$ for which $W(f)=0,$ for all $W\in \lambda^{\perp}.$ Therefore \begin{equation}\mathcal{L}_{[k,X]}g=df\otimes_{S} k^{\natural}+f\mathcal{L}_{k}g\end{equation} is of boost-order $\leq -2$ and hence 
\begin{equation}
    \mathcal{L}_{k}\mathcal{L}_Xg=\mathcal{L}_{[k,X]}g-\mathcal{L}_{X}\mathcal{L}_{k}g
\end{equation}
is also of boost-order $\leq -2$ w.r.t. $\lambda,$ since $\mathcal{L}_{X}$ preserves algebraic structure. It is now easy to see that $\mathcal{L}_{k}\mathcal{L}_Xg$ is of boost-order $\leq -2$ w.r.t. $\lambda$ for all Kundt vector fields $k$. By proposition \ref{EasySpiPreservation} it follows that $X$ preserves the spi's of $g$ iff. $(\mathcal{L}_{k})^2\mathcal{L}_{X}g=0,$ for each Kundt vector field $k.$

$"\Rightarrow"$
Suppose that $(M,g,\lambda)$ is CSI. Then tautologically, if $X$ is an algebra preserving nil-Killing vector field w.r.t. $\lambda$, it preserves spi's, and therefore by the above discussion $X$ satisfies $i),ii)$ and $iii).$ The result therefore follows from proposition \ref{CSILocallytransitive}.

$"\Leftarrow"$ This is a consequence of proposition \ref{EasySpiPreservation}.
\end{proof}

 Lastly, our discussion in section \ref{spipreservingdefs} has among other shown the following: \emph{A nil-Killing vector field $X$ on a degenerate Kundt spacetime preserves spi's provided that $\mathcal{L}_{X}\nabla^{m}Rm$ is of type $III$ with respect to the given null-distribution, for all $m\geq 0.$ }
 
 Hence the existence of a locally transitive collection of such vector fields about any given point would ensure the metric to be CSI. In future work we shall attempt to show that the converse also holds: That a given degenerate Kundt metric which is CSI has such a locally transitive collection about any given point. The next result taken together with proposition \ref{CSILocallytransitive} shows that this is true in dimension three:
\begin{proposition}
Suppose that $(M,g,\lambda)$ is a three-dimensional degenerate Kundt spacetime. Let $X$ be an algebra preserving nil-Killing vector field w.r.t. $\lambda$. $X$ preserves spi's iff. $\mathcal{L}_{X}\nabla^{m}Rm$ is of type $III$ w.r.t. $\lambda,$ for all $m\geq 0.$
\end{proposition}
\begin{proof}
$"\Rightarrow"$ If $X$ preserves spi's, then it follows from work in \cite{SCPI} that we can find a Kundt frame $\{k,l,m\}=\{e_1,e_2,e_3\}$ such that the boost-weight zero components of $\nabla^{s}Rm$, for $s\geq0,$ can be expressed in terms of spi's, and are therefore constant along the integral curves of $X.$ By the nil-Killing property
\begin{equation}
    \mathcal{L}_{X}g(m,m)=-2g(\mathcal{L}_{X}m,m)=0,
\end{equation}
from which it follows that $\mathcal{L}_{X}m\in \lambda.$
Since $X$ is algebra preserving there exists a smooth function $f$ such that $\mathcal{L}_{X}k=fk.$
Therefore
\begin{equation}
0=\mathcal{L}_{X}g(k,l)=-g(fk,l)-g(k,\mathcal{L}_{X}l)=-f-g(k,\mathcal{L}_{X}l),
\end{equation}
showing that 
\begin{equation}
    \mathcal{L}_{X}l=-fl + ak+bm,
\end{equation}
for some smooth function $a,b.$ It follows that
\begin{equation}
    \mathcal{L}_{X}(k\otimes l)=fk\otimes l +k\otimes (-fl + ak +bm)=a(k\otimes k)+b(k\otimes m) ,
\end{equation}
which is of boost-order $-1$ w.r.t. $\lambda.$ Thus given any tensor of the form
$
    A(e_{i_{1}}\otimes \cdots \otimes e_{i_k})
$
of boost-weight $0$, such that $X(A)=0$, then $\mathcal{L}_{X} A(e_{i_{1}}\otimes \cdots \otimes e_{i_k})$ is of boost-order $\leq -1,$ w.r.t. $\lambda.$ This implies that $\mathcal{L}_{X}\nabla^{s}Rm$ is of type $III$ w.r.t. $\lambda$, for all $s\geq 0.$

$"\Leftarrow"$ This follows from the discussion in the previous section.
\end{proof}

\section{Conclusion}
In this paper we have studied smooth deformations and transformations of  metrics in the direction of type $III$ tensors. The transformations of this nature  are given by the algebra preserving nil-Killing vector fields. We derive the local coordinate form of all such vector fields.

We classify Kundt spacetimes as those for which there exists a null-distribution, containing vector fields about each point which are nil-Killing with respect to the distribution, and whose orthogonal complement is integrable. Moreover we characterize the Kundt spacetimes having a locally homogeneous transverse metric in terms of the existence of a locally transitive collection of nil-Killing vector fields belonging to the orthogonal complement to the null-distribution. Lastly, we present a classification of degenerate Kundt metrics.

For degenerate Kundt spacetimes we provide a theorem which gives a classification of tensors whose boost-order remains unaltered under covariant derivatives. This characterization together with the use of deformation equations allows for a result giving conditions for a deformation to preserve spi's for degenerate Kundt spacetimes. 

We use this result to produce a class of nil-Killing vector fields which preserve scalar curvature invariants. In turn this is used in order to give a new proof that under an assumption of uniqueness, the deformations provided in \cite{SCPI,CharBySpi} leave the orbit of the metric. 

Lastly we discuss Kundt-CSI metrics in terms of nil-Killing vector fields. We show that each Kundt-CSI metric has locally transitive collections of nil-Killing vector fields about any point. Moreover we characterize Kundt-CSI metrics with recurrent null vector fields in terms of the existence of locally transitive collections of nil-Killing vector fields satisfying some added assumptions.

\section*{Acknowledgements}
I would like to thank  Lode Wylleman, Boris Kruglikov and my advisor Sigbjørn Hervik for helpful discussions concerning this project.

\medskip

\bibliography{NilKilling}{}
\bibliographystyle{plain}
\end{document}